\documentclass[10pt,a4paper]{amsart}

\usepackage{enumerate}
\usepackage{amsfonts}
\usepackage{amssymb}
\usepackage{amsthm}
\usepackage{amsmath}
\usepackage{textcomp}
\usepackage{mathrsfs}
\usepackage{calligra}
\usepackage{epigraph}
\usepackage[T1]{fontenc}
\usepackage{array}
\usepackage{fancyhdr}
\usepackage{setspace}
\usepackage{multirow}
\usepackage{tabularx}
\usepackage{rotating}
\usepackage{longtable}
\usepackage{mathscinet}
\usepackage{mdwlist}
\usepackage{upgreek}
\usepackage{cleveref}
\usepackage{marginfix}

  \renewenvironment{thebibliography}[1]{%
    \begin{oldthebibliography}{#1}%
      \setlength{\parskip}{0ex}%
      \setlength{\itemsep}{0ex}%
  }%
  {%
    \end{oldthebibliography}%
  }

 \def\a{\alpha}                             
                 \def\i{\iota}                
 \def\g{\gamma}                              
                \def\lmb{\lambda}              \def\sg{\sigma}
 \def\e{\epsilon}                               \def\t{\tau}
        \def\n{\nu}                  
                 \def\x{\xi}                  \def\ph{\phi}

                                                           \def\ps{\psi}

			\def\bbR{\mathbb{R}}
			\def\bbN{\mathbb{N}}
			
			\def\bbQ{\mathbb{Q}}

           \newcommand{\less}{\setminus}
           
           \newcommand{\un}{\cup}
           \newcommand{\Un}{\bigcup}

              \def\ex{\exists}
                       
           \def\sset{\subseteq}        

\def\mc{\mathcal}
\def\mfk{\mathfrak}
\DeclareMathAlphabet{\mathpzc}{OT1}{pzc}{m}{it}

\def\Ran{\bbR_{\textnormal{an}}}

\newcommand{\rst}[1]{\!\!\upharpoonright_{#1}}
\newcommand{\norm}[1]{\left|\left|#1\right|\right|}

\newcommand{\md}[1]{\left|#1\right|}

\newcommand{\tr}[1]{#1^{\textrm{trans}}}

\renewcommand {\d}[2] {\frac{\partial #1}{\partial #2}}

\newcommand\addtag{\refstepcounter{equation}\tag{\theequation}}

\newcommand{\blob}{\times}

\theoremstyle{definition}
\newtheorem{defn}{Definition}[section]

\theoremstyle{plain}
\newtheorem{theorem}[defn]{Theorem}

\newtheorem{lemma}[defn]{Lemma}
\newtheorem{proposition}[defn]{Proposition}
\newtheorem{corr}[defn]{Corollary}

\theoremstyle{remark}
\newtheorem{remark}[defn]{Remark}

\newtheorem{notation}[defn]{Notation}

\numberwithin{equation}{defn}

\begin{document}

\title[Effective Pila--Wilkie bounds for unrestricted Pfaffian surfaces]{Effective Pila--Wilkie bounds for unrestricted Pfaffian surfaces}

\author{Gareth O. Jones}
\thanks{First author supported by UK Engineering and Physical Sciences Research Council (EPSRC) Grants EP/J01933X/1 and EP/N007956/1.}
\address{School of Mathematics, University of Manchester, Oxford Road, Manchester M13 9PL, UK}
\author{Margaret E. M. Thomas}
\thanks{Second author supported by German Research Council (DFG) Grant TH 1781/2-1; the Zukunftskolleg, Universit\"at Konstanz; the Fields Institute for Research in Mathematical Sciences, Toronto, Canada; the Ontario Baden--W\"urttemberg Foundation; and Canada Natural Sciences and Engineering Research Council (NSERC) Discovery Grant RGPIN 261961.}
\address{Department of Mathematics and Statistics, McMaster University, 1280 Main Street West, Hamilton, Ontario, Canada, L8S 4K1}
\address{Zukunftskolleg, Fachbereich Mathematik und Statistik, Fach 216, Universit\"atsstra\ss e 10, Universit\"at Konstanz, D-78457 Konstanz, Germany}
\begin{abstract}
We prove effective Pila--Wilkie estimates for the number of rational points of bounded height lying on certain surfaces defined by Pfaffian functions. The class of surfaces to which our result applies includes, for instance, graphs of unrestricted Pfaffian functions defined on the plane.
\end{abstract}
\keywords{Pfaffian functions, effectivity, rational points, the Pila--Wilkie Theorem}
\subjclass[2010]{03C64 (Primary), 14G05, 14Q20 (Secondary)}
\maketitle
\section{Introduction}
\label{sec:intro}
Suppose that $X \subseteq \mathbb{R}^{n}$ is a set definable in an o-minimal expansion of the real field. The Pila--Wilkie Theorem \cite{PilWil-06} provides a subpolynomial bound on the number of rational points of bounded height lying on the transcendental part of such a set $X$. More precisely, let $X^{\textrm{alg}}$, its \emph{algebraic part}, be the union of all connected, infinite, semi-algebraic subsets of $X$, and let $X^{\textrm{trans}}$, its \emph{transcendental part}, be the complement of $X^{\textrm{alg}}$ in $X$. Given a rational point $\overline{q} = (\frac{a_1}{b_1},\ldots,\frac{a_{n}}{b_{n}}) \in \mathbb{Q}^{n}$, where $\gcd{(a_{i},b_{i})}=1$ for each $i=1,\ldots,n$, the height of $\overline{q}$ is $H(\overline{q}) = \max_{1\leq i \leq n}\{|a_i|,|b_i|\}$. The Pila--Wilkie Theorem states that, for any positive real number $\epsilon$ and any $T\geq 1$, there are at most $cT^{\epsilon}$ rational points of height at most $T$ lying on $X^{\textrm{trans}}$, where $c$ is a positive real number depending on $X$ and $\epsilon$.

In fact, Pila and Wilkie proved several stronger statements, including the provision of a constant $c$ which is uniform across the fibres of a definable family $Z \subseteq \mathbb{R}^{m} \times \mathbb{R}^{n}$. Analogous bounds were moreover established by Pila in \cite{Pil-09-alg} for algebraic points of bounded height and degree, where the constant $c$ depends on $X$, $\epsilon$ and a bound $k$ on the degree of the algebraic points. Pila and Habegger then extended the result further \cite{HabPil-16}. These results all share a common feature with the earlier work of Pila \cite{Pil-04, Pil-05} on subanalytic surfaces, in that the proof does not provide a method for computing the constant $c$ effectively in terms of $\epsilon$, some definition of $X$ and, if applicable, $k$. Indeed, at the level of generality of the class of all sets $X$ definable in an o-minimal expansion of the real field, it is not so clear what it might mean for the constant to be effective.

Despite this, the question of the existence of an effective constant remains valid in certain cases, and indeed it is an interesting question in view of the many applications of the Pila--Wilkie Theorem to diophantine geometry. One setting in which there is certainly a reasonable measure of complexity is in a reduct of $\Ran$, the expansion of the real field by all restricted analytic functions, in which all functions added to the real field are assumed to satisfy some reasonable differential equations. Binyamini \cite{Bin-17} has recently shown a rather general result in this direction.

The main result of this paper goes beyond the setting of reducts of $\Ran$.
We give an effective version of the Pila--Wilkie Theorem for surfaces, under the assumption that the surface has a certain definition in terms of Pfaffian functions. An analytic function $f$ on an open subset of $\bbR^n$ is Pfaffian if it satisfies a triangular system of polynomial differential equations. We give precise definitions of the setting in which we work in the next section. For now, note that Pfaffian functions have a natural measure of complexity given by the dimension of the open set, the number of equations in the differential equation system and the degrees of the polynomials involved. A precise statement of our main result (Corollary \ref{cor:main_surface_unres}) will be given and proved in Section \ref{sec:counting}. For now, in the interests of simplicity and by way of example, we state a theorem which is an immediate consequence of our main result.

\begin{theorem} \label{thm:unrescountingintro}
Suppose that $f:\bbR^2\to \bbR$ is Pfaffian. Let $\epsilon>0$. There exists a positive constant $c$ depending only on the complexity of $f$ and on $\epsilon$, and effectively computable from them, with the following property. For all $T\geq 1$, the transcendental part of the graph of $f$ contains at most $cT^\epsilon$ rational points of height at most $T$.
\end{theorem}

There are two improvements in the constant obtained here over that which the Pila-Wilkie Theorem provides for such functions. Of course, one improvement is that the constant here is effective. The other is that it is uniform across the class of all Pfaffian functions of the same complexity. (We obtain a similar uniformity in our main result.) In fact, a high level of uniformity is crucial for our proof in this unrestricted setting. We straightforwardly obtain our main result (and, therefore, Theorem \ref{thm:unrescountingintro}) from a result for surfaces defined by restricted Pfaffian functions that is uniform in the restriction taken (Theorem \ref{thm:main_surface_res_inbox}); most of the work in this paper goes into proving that result.  This uniformity across complexity is new, even for the original ineffective setting of the Pila--Wilkie Theorem. We contrast this with the recent work of Binyamini \cite{Bin-17} alluded to above, which provides an effective constant for sets of all dimensions described by restricted Noetherian functions. These are functions defined in the same way as Pfaffian functions but without the triangularity assumption on the system of differential equations. Our main result only applies to surfaces, but these need not be restricted. Thus there is some overlap in the cases covered by our main result and by Binyamini's result specialised to surfaces, but neither setting contains the other. In the case that there is overlap, the constant obtained in \cite{Bin-17} also depends on input data other than the complexity (as we define it here), such as the coefficients of the polynomials appearing in the system of differential equations and the restriction of the function taken. The effective uniformity that we obtain is a feature which could potentially be exploited in applications, as we discuss below.

In light of recent work of Binyamini and Novikov \cite{BinNov-17-WC} which establishes an improvement to the Pila--Wilkie Theorem, in the form of a polylogarithmic bound, for restricted complex functions with real and imaginary parts which are Pfaffian, it is also worth noting that our main results here apply whenever the functions involved are real Pfaffian, without further conditions being imposed.

As mentioned above, the Pila-Wilkie Theorem and its later variants have been applied to a variety of questions in diophantine geometry. In particular, it features as part of the Pila--Zannier method, based on their proof of the Manin--Mumford Conjecture \cite{PilZan-08}, for addressing problems on unlikely intersections associated with the names Manin, Mumford, Andr\'e, Oort, Zilber, and Pink. For a sample of such applications, we refer the reader to \cite{Zan}, \cite{Sca-12}, \cite{Pil-14} and \cite{JonWil-15}. In fact, we could hope for more than effectivity, as we mentioned above, and we now discuss this further. The abelian logarithms used in the Pila--Zannier proof of the Manin--Mumford Conjecture are known, by an observation of Macintyre \cite{Mac-08}, to be Pfaffian when restricted to appropriate domains. In the elliptic case, this has been developed by Schmidt and the first author \cite{JonSch-17} to obtain an explicit definition of the Weierstrass $\wp$-function associated to an elliptic curve, when the function is restricted to a fundamental domain. With a uniform choice of this domain, the complexity of the definition is independent of the elliptic curve. It is plausible that this can be extended to the abelian case (perhaps with effective rather than explicit estimates). Now, the polynomial Galois bounds involved in the Pila--Zannier proof, when everything is defined over a number field, are conjectured to depend only on the dimension of the abelian variety. If an effective form of this could be proved then, combined with Corollary \ref{cor:main_surface_unres} and the expected uniform definition, the Pila--Zannier strategy could lead to a Manin--Mumford result for curves in abelian varieties with a bound that is effective and that is independent of the abelian variety (when its dimension is fixed). A less dramatic (but also less conjectural!) example of Corollary \ref{cor:main_surface_unres} being applied is presently being worked out by Schmidt and the first author \cite{JonSch-18-relMM}, in the context of the relative Manin--Mumford Conjecture in the case that the group involved is a family of extensions by $\mathbb{G}_m$ of a fixed elliptic curve $E$, with everything defined over the algebraic numbers (this is the `semi-constant' case of \cite{BMPZ}; see also \cite{BerSch-18} for a recent extension beyond the Manin-Mumford setting). The final bound obtained is effective and does not depend on the height of the coefficients of the polynomials defining the curve. The counting result is both effective and uniform in the elliptic curve. This seems to require the uniformity in domain that we establish here. For effective uniformity in the context of additive extensions of a fixed elliptic curve, also via Pfaffian functions, but without the need for Galois bounds or counting, see \cite{JonSch-18-MM}.

Most of the work in this paper is invested in proving a uniform result in the restricted case (Theorem \ref{thm:main_surface_res_inbox}) which we briefly mentioned above. Our proof broadly follows the same outline as that of the proof of the Pila--Wilkie Theorem. In particular, we use a parameterization result, a covering of surfaces by finitely many subsets described by functions with controlled derivatives. Here we cannot appeal directly to the o-minimal version of the parameterization of Yomdin and Gromov \cite{Yom-87-vge, Yom-87-Ckr, Gro-87} that was proved by Pila and Wilkie. Indeed, this result is one of the main sources of ineffectivity in the Pila--Wilkie Theorem, for it involves the use of the compactness theorem (of first-order logic). Our main contribution is an effective version of this parameterization result in the restricted case. To prove this, we first prove a uniform parameterization result for families of curves in the restricted case, where the base of the family is an interval in the real line. The details of these results are rather technical, so we defer the statements to Sections \ref{sec:curves} and \ref{sec:surfaces}. A complication in proving our parameterization results is that we work in a wider setting than that of the graphs of (restricted) Pfaffian functions alone, due to the inductive nature of the proofs and the fact that a key technique we use is an effective decomposition of zero sets, given by Lemma \ref{lem:effzerosetdecomp}, which uses functions which are not necessarily Pfaffian. Instead we consider classes of functions implicitly defined by (restricted) Pfaffian functions, $IP$ and $IRP$, which have important closure properties. For instance, they are closed under differentiation, a property which some other natural classes do not have. In addition, the class $IRP$ has equations that hold at the boundary, and this is a crucial feature for us in working with the limits of functions (see, for example, Lemma \ref{lem:limit}), which is critical at several stages of the proofs.  In working with these functions we exploit an effective stratification result of Gabrielov and Vorobjov \cite{GabVor-95} on several occasions. 

This paper provides a self-contained presentation of an effective and uniform form of the Pila--Wilkie Theorem for unrestricted surfaces in our setting. It is expected that an extension of these methods would yield an analogous result in higher dimensions; however, given the rather technical nature of the proofs, the case of surfaces presents the most appropriate setting for outlining the key ideas involved. Moreover, this case is already of interest in its own right, and is sufficient for the application in the context of the relative Manin--Mumford Conjecture in \cite{JonSch-18-relMM} that is outlined above.

We give precise definitions of all of the concepts introduced here in Sections \ref{sec:prelims} and \ref{sec:curves}. In Section \ref{sec:lemmas} we prove some auxiliary lemmas concerning the behaviour of functions which are implicitly defined from restricted Pfaffian functions. Our proofs of parameterization results in the restricted setting follow in Sections \ref{sec:curves} and \ref{sec:surfaces} (an effective uniform parameterization result for families of curves in Section \ref{sec:curves}, and an effective parameterization result for surfaces in Section \ref{sec:surfaces}). Finally we state and prove the diophantine results of this paper in Section \ref{sec:counting}.

\emph{Acknowledgements.} The authors are grateful to Harry Schmidt for an inspiring conversation which provided a key insight in the unrestricted case.

\section{Preliminaries on Pfaffian functions}
\label{sec:prelims}
We begin by making precise our terminology related to effectivity. When we say that a quantity $N$ is \emph{effective in} certain parameters, we will mean that $N$ is effectively computable from those parameters, and say that $N$ is \emph{bounded effectively in} certain parameters when there is a bound on $N$ which is effectively computable from those parameters.

We now make precise our definitions involving Pfaffian functions and recall various results about them which we will need in later proofs.

\begin{defn} \label{def:Pfaffian}
Let $n$, $r$, $\alpha$, $\beta$ be non-negative integers. A sequence $f_1,\ldots,f_r:U\to \bbR$ of analytic functions on an open set $U\subseteq \bbR^n$ is said to be a \emph{Pfaffian chain} of \emph{order} $r$ and \emph{degree} $\alpha$ if there are polynomials $P_{i,j}\in \bbR [X_1,\ldots,X_{n+j}]$ of degree at most $\alpha$ such that
\begin{equation*}
df_j=\sum_{i=1}^n P_{i,j}(\bar{x},f_1(\bar{x}),\ldots,f_j(\bar{x}))dx_i ,\qquad \textrm{ for all } i=1,\ldots,r \textrm{ and }j=1,\ldots,n.
\end{equation*}
Given such a chain, we say that a function $f:U\to \bbR$ is \emph{Pfaffian of order} $r$ and \emph{degree} $\langle\alpha,\beta\rangle$ with chain $f_1,\ldots,f_r$, if there is a polynomial $P \in \bbR[X_1,\ldots,X_n,Y_1,\ldots,Y_r]$ of degree at most $\beta$ such that $f(\bar{x})=P(\bar{x},f_1(\bar{x}),\ldots,f_r(\bar{x}))$.

Let $B$ be a positive real number. We say that a Pfaffian function has \emph{complexity} at most $B$ if $n,r, \alpha$ and $\beta$ are all at most $B$.
\end{defn}

The following theorem of Khovanskii is the foundation for the theory of Pfaffian functions. It provides an effective bound on the number of connected components of a Pfaffian variety \cite{Kho-80, Kho-91}, when the domain of the function is sufficiently simple.

Given functions $g_1,\ldots,g_k:\bbR^n\to\bbR$, as is customary we write $V(g_1,\ldots,g_k)=\{ \bar{x}\in\bbR^n : g_1(\bar{x})=\cdots = g_k(\bar{x})=0\}$ (and for $k=0$ this is $\bbR^n$). Throughout this paper, a `box' is always understood to be a product of bounded intervals.
\begin{theorem} \label{thm:Khovanskii}
Let $n$, $k$ be non-negative integers, let $B$ be a positive real number and let $W$ be an open box in $\bbR^n$. Suppose that $g_1,\ldots,g_k:W\to\bbR$ are Pfaffian functions with a common chain and complexities bounded by $B$. There exists a positive real number $B'$ which is bounded effectively in $B$ such that the number of connected components of the variety $V(g_1,\ldots,g_k)$ is bounded by $B'$.
\end{theorem}

Our results concern functions implicitly defined from Pfaffian functions, in both restricted and unrestricted settings. These two frameworks need to be defined precisely. We begin with the unrestricted setting.
 
\begin{defn}[IP] \label{def:IP}
Let $n$, $k$ be non-negative integers and let $U\subseteq \bbR^n$ be an open set.
Let $B$ be a positive real number. We say that a function $f:U\to \bbR$ is in the class $IP(B)$, and say that $f$ is \emph{implicitly defined from Pfaffian functions of complexity at most $B$}, if there exist a positive integer $m$, a product of open intervals $V\subseteq \bbR^{n+m}$, an auxiliary map $F=\langle f_1,\ldots,f_m \rangle:U\to \bbR^{m}$, and Pfaffian functions $p_1,\ldots,p_{m}: V\to \bbR$ with a common chain and complexities bounded by $B$, such that $f_1=f$ and the following hold:
\begin{itemize}
\item[(i)] $\text{graph}(F)$ is contained in $V$;
\item[(ii)] $p_i(x,F(x))=0$, for $i=1,\ldots,m$ and all $x \in U$;
\item[(iii)] $\det \left(\d{(p_1,\ldots,p_m)}{(x_{n+1},\ldots,x_{n+m})}\right) (x,F(x))\ne 0$, for all $x \in U$.
\end{itemize}
We say that a map $g:U\to\bbR^{k}$ is in the class $IP(B)$ if each of its coordinate functions $g_{j}\colon U\to \bbR$ is in $IP(B)$, for $j = 1,\ldots, k$. We say that a function $f:U\to \bbR$ is in the class $IP$, and say that it is \emph{implicitly defined from Pfaffian functions}, if $f$ is in the class $IP(B)$ for some positive real number $B$.
\end{defn}

We need an analogue of this definition in which the Pfaffian functions involved are restricted. For this we use the standard device of an algebraic analytic isomorphism from $\bbR$ to $(-1,1)$.
We must first fix some notation. Let $\vartheta\colon \bbR \to (-1,1)$ be the analytic bijection given by $\vartheta(x)=\frac{x}{\sqrt{1+x^2}}$. Note that $\vartheta$ is Pfaffian (on $\bbR$) with chain $x\mapsto \frac{1}{\sqrt{1+x^2}}$ and complexity at most $3$. Its inverse $\vartheta^{-1} = \frac{x}{\sqrt{1-x^2}}$ is also analytic and Pfaffian (on $(-1,1)$) with chain $x\mapsto \frac{1}{\sqrt{1-x^2}}$ and also has complexity at most $3$. For any non-negative integer $n$ we also write $\vartheta:\bbR^n\to (-1,1)^n$ for the map $\langle x_1, \ldots, x_n \rangle \mapsto \langle\vartheta(x_1),\ldots,\vartheta(x_n)\rangle$, i.e. with the function $\vartheta$ applied coordinate-wise.

\begin{defn}[IRP] \label{def:IRP}
Let $n$, $k$ be non-negative integers and let $U\subseteq \bbR^n$ be an open set.
Let $B$ be a positive real number. We say that a function $f:U\to \bbR$ is in the class $IRP(B)$, and say that $f$ is \emph{implicitly defined from restricted Pfaffian functions of complexity at most $B$}, if there exist a positive integer $m$, a product of open intervals $V\subseteq \bbR^{n+m}$, an auxiliary map $F=\langle f_1,\ldots,f_m \rangle:U\to \bbR^{m}$, and Pfaffian functions $p_1,\ldots,p_{m}: V\to \bbR$ with a common chain and complexities bounded by $B$, such that $f_1=f$ and the following hold:
\begin{itemize}
\item[(i)] the closure of $\vartheta(\text{graph}(F))$ is contained in $V$;
\item[(ii)] $p_i(\vartheta(x,F(x)))=0$, for $i=1,\ldots,m$ and all $x \in U$;
\item[(iii)] $\det \left(\d{(p_1,\ldots,p_m)}{(x_{n+1},\ldots,x_{n+m})}\right) (\vartheta(x,F(x)))\ne 0$, for all $x \in U$.
\end{itemize}
We say that a map $g:U\to\bbR^{k}$ is in the class $IRP(B)$ if each of its coordinate functions $g_{j}\colon U\to \bbR$ is in $IRP(B)$, for $j=1,\ldots,k$. We say that a function $f:U\to \bbR$ is in the class $IRP$, and say that it is \emph{implicitly defined from restricted Pfaffian functions}, if $f$ is in the class $IRP(B)$ for some positive real number $B$. 
\end{defn}
Given a positive real number $B$ and a function $f$ lying in $IP(B)$, respectively $IRP(B)$, clearly all of the coordinate functions of any auxiliary map $F$ witnessing this membership must also themselves lie in $IP(B)$, respectively $IRP(B)$. We simply require $f_1 = f$ for convenience.

The class $IRP$ has useful closure properties, as we show below. It is, for instance, closed under differentiation. In addition, we have `equations at the boundary', another feature that will be exploited in our proofs. This combination of properties does not seem to hold for some other, at first sight natural, classes.

Since $\vartheta \colon \bbR \to (-1,1)$ is itself a Pfaffian function, and moreover is an analytic isomorphism, it is easy to see that, for any positive real number $B$, the class $IRP(B)$ is contained in the class $IP(B)$.
The following lemma shows that functions in $IP(B)$ satisfying certain additional conditions also lie in $IRP(B')$, for a complexity $B'$ which is bounded effectively in $B$.
\begin{lemma} \label{lem:resIP}
Let $n$ be a non-negative integer and let $B$ be a positive real number. Suppose that $U \sset \bbR^n$ is a bounded open set and that $f\colon U \to \bbR$ is a function lying in $IP(B)$, witnessed by an auxiliary map $F \colon U \to \bbR^m$, for some positive integer $m$, with $V\subseteq \bbR^{n+m}$ the domain of the associated Pfaffian functions. Suppose that the map $F$ is bounded and that the closure of the graph of $F$ is contained in $V$. There exists a positive real number $B'$ bounded effectively in $B$ such that $f$ lies in $IRP(B')$.
\end{lemma}
\begin{proof}
Let $p_{1},\ldots,p_{m}\colon V \to \bbR$ be the Pfaffian functions (whose complexities are at most $B$) witnessing that $f$ lies in $IP(B)$. The closure of the graph of $F$ is a compact set and so there exists an open box $W$ with $\overline{\text{graph}(F)} \sset W$ and $\overline{W} \sset V$. Set $V':=\vartheta(W)$; this is likewise an open box and it contains the closure of $\vartheta(\text{graph}(F))$. For each $i=1,\ldots,m$, define the function $q_{i}\colon V'\to\bbR$ by setting $q_{i}(x_{1},\ldots,x_{n+m}) = p_{i}(\vartheta^{-1}(x_{1}),\ldots,\vartheta^{-1}(x_{n+m}))$. These functions are all Pfaffian on $V'$, with a complexity $B'$ which is bounded effectively in $B$, and clearly we have $q_i(\vartheta(x,F(x)))= p_i(x,F(x))=0$, for all $i=1,\ldots,m$. 
Further, a quick calculation (using the fact that the derivative of $\vartheta^{-1}$ doesn't vanish) shows that the non-singularity condition also holds. 
\end{proof}

We will also make use of the following corollary to this result, which tells us that certain restrictions of functions lying in $IP$ also lie in $IRP$, with a complexity that is independent of the restriction taken.

\begin{corr}\label{cor:resIP=IRP}
Let $n$ be a non-negative integer, let $U \sset \bbR^n$ be an open set and let $B$ be a positive real number. Suppose that $f\colon U \to \bbR$ is a function lying in $IP(B)$. There exists a positive real number $B'$, bounded effectively in $B$, such that, if $W$ is any bounded open set whose closure lies in $U$, then $f\rst{W}$ lies in $IRP(B')$.
\end{corr}
\begin{proof}

Let $m$ be a positive integer, $V$ a product of open intervals, $F\colon U \to \bbR$ an auxiliary map and $p_{1},\ldots,p_{m}\colon V\to \bbR$ a collection of Pfaffian functions, whose complexities are at most $B$, which all together witness that $f$ lies in $IP(B)$. Since the map $F$ is continuous on $U$, it is continuous on the closure of $W$, and hence bounded on $W$. We may now apply Lemma \ref{lem:resIP} to obtain the result.
\end{proof}

For the results in this paper we will frequently make use of the fact that certain named functions lie in the class $IRP$, and that this class is closed in an effective way under various common operations. It is straightforward to check this in the case of composition, i.e. that, given a positive real number $B$ and maps $f, g \in IRP(B)$ such that the composition $g \circ f$ is well defined, there exists a positive real number $B'$ bounded effectively in $B$ such that $g \circ f \in IRP(B')$. 

The class $IRP$ is also closed in this way under taking multiplicative inverses, by combining closure under composition together with the fact that the function $f\colon (0,\infty) \to \bbR$ given by $f(x) = \frac{1}{x}$ lies in the class $IRP$. To see this latter fact, let $f_2 \colon (0,\infty)\to \bbR$ be given by $f_{2}(x)=\frac{1}{\sqrt{1+x^2}}$ and let $F = \langle f, f_{2} \rangle$, and let $p_{1},p_{2}\colon \bbR^{2} \times (-1,1) \to \bbR$ be the Pfaffian functions given by
\begin{align*}
p_{1}(x_1,x_2,x_3) &= x_{1}^{2} + \vartheta^{-1}(x_{3})^{2} - 1,\\
p_{2}(x_1,x_2,x_3) &= \vartheta^{-1}(x_{3})^{2} - x_2.
\end{align*}
Note that $f_2$ has image $(0,1]$, so $\vartheta(f_{2}(x))$ is bounded away from $\pm 1$ on $\bbR$, and hence the closure of $\vartheta(\text{graph}(f))$ is contained in $\bbR^{2} \times (-1,1)$. Clearly we have \[p_{i}(\vartheta(x),\vartheta(f(x)),\vartheta(f_{2}(x))) = 0,\] for $i=1,2$, and from this we see that \[\frac{\partial p_1}{\partial x_3}(\vartheta(x),\vartheta(f(x)),\vartheta(f_{2}(x))) = 2\vartheta^{-1}(x_{3})\frac{d}{dx_{3}}\vartheta^{-1}(x_{3}).\] This does not vanish at a point of the form $\vartheta(f_{2}(x))$, and so it is straightforward to see that the necessary singularity condition is also satisfied.

In order to see that products and sums of functions lying in $IRP$ also lie in $IRP$ with an effectively bounded complexity, it is enough, when combined with closure under composition, to see that the functions $\blob\colon \bbR^{2} \to \bbR$, given by $\langle x_1,x_2\rangle \mapsto x_{1}\cdot x_{2}$, and $+\colon \bbR^{2} \to \bbR$, given by $\langle x_1,x_2\rangle \mapsto x_{1} + x_{2}$, lie in $IRP$. This is witnessed by the following. Let $F_{\blob},F_{+}\colon \bbR^2 \to \bbR^{4}$ be given by
\begin{align*}
F_{\blob}(x_1,x_2) &= \left\langle x_{1} \cdot x_{2}, \frac{1}{\sqrt{1+x_{1}^{2}}}, \frac{1}{\sqrt{1+x_{2}^{2}}}, \frac{1}{\sqrt{1+(x_{1}\cdot x_{2})^{2}}}\right\rangle,\\
F_{+}      (x_1,x_2) &= \left\langle x_{1}  +    x_{2}, \frac{1}{\sqrt{1+x_{1}^{2}}}, \frac{1}{\sqrt{1+x_{2}^{2}}}, \frac{1}{\sqrt{1+(x_{1} +    x_{2})^{2}}}\right\rangle,
\end{align*}
and let $q_{1},\ldots,q_{4}, r_{1},\ldots,r_{4} \colon \bbR^3 \times (-1,1)^3 \to \bbR$ be the Pfaffian functions given by
\begin{align*}
q_{1}(x_1,\ldots,x_6) &=  r_{1}(x_1,\ldots,x_6) = \quad x_{1}^{2} + \vartheta^{-1}(x_{4})^{2} - 1,\\
q_{2}(x_1,\ldots,x_6) &=  r_{2}(x_1,\ldots,x_6) = \quad x_{2}^{2} + \vartheta^{-1}(x_{5})^{2} - 1,\\
q_{3}(x_1,\ldots,x_6) &=  r_{3}(x_1,\ldots,x_6) = \quad x_{3}^{2} + \vartheta^{-1}(x_{6})^{2} - 1,\\
q_{4}(x_1,\ldots,x_6) &= \quad x_{3}\vartheta^{-1}(x_{4})\vartheta^{-1}(x_{5})\vartheta^{-1}(x_{6}) - x_{1}x_{2}\vartheta^{-1}(x_{6})^{2},&\\
r_{4}(x_1,\ldots,x_6) &= \quad x_{3}\vartheta^{-1}(x_{4})\vartheta^{-1}(x_{5})\vartheta^{-1}(x_{6}) - (x_{1}\vartheta^{-1}(x_{5})+x_{2}\vartheta^{-1}(x_{4}))\vartheta^{-1}(x_{6})^{2}.&
\end{align*}
That the functions $\blob$ and $+$ lie in $IRP$ is witnessed by $F_{\blob}$ together with $q_{1},\ldots,q_{4}$, and by $F_{+}$ together with $r_{1}\ldots,r_{4}$, respectively. 
It is easy to see from this that polynomials in the functions of $IRP$ also lie in $IRP$ with effectively bounded complexities, and hence that $\vartheta' \colon \bbR \to \bbR$, $\vartheta'(x) = \left(\frac{1}{\sqrt{1+x^2}}\right)^{3}$ lies in $IRP$ (since it follows from the implicit definitions of $\langle x_1,x_2 \rangle \mapsto x_{1}\cdot x_{2}$ and $\langle x_1,x_2 \rangle \mapsto x_{1}+x_{2}$ that $x\mapsto \frac{1}{\sqrt{1+x^2}}$ also lies in $IRP$ with effectively bounded complexity). With the information now at hand, it follows readily that the class $IRP$ is closed under taking partial derivatives, with effectively bounded complexities, using the observation that $\vartheta'$ is moreover nowhere zero.

We may further observe that the functions $\vartheta \colon \bbR \to (-1,1)$ and $\vartheta^{-1}\colon (-1,1)\to\bbR$ also lie in $IRP$ with complexities at most $3$. In each case, take $m=1$ and $V = \bbR^{2}$. Then the former is witnessed by the Pfaffian function $p(x_1,x_2) = \vartheta(x_1) - x_2$, the latter by the Pfaffian function $q(x_1,x_2) = x_1 - \vartheta(x_2)$.

Finally, if $B$ is a positive real number and $f\colon I \to J$ is a bijection lying in $IRP(B)$, for $I$, $J \subseteq \bbR$, such that $f'(x)\neq 0$, for all $x \in I$, then it is straightforward to check that the inverse $f^{-1}\colon J\to I$ also lies in $IRP(B)$.

A key result we shall need is an effective stratification theorem due to Gabrielov and Vorobjov \cite{GabVor-95}. The sets involved are more general than varieties and therefore we need the following definitions.
\begin{defn}
An \emph{elementary semi-Pfaffian set} $X$ is a set of the form
\begin{displaymath}
\{ \bar{x} \in U : g_1(\bar{x})=\cdots = g_k(\bar{x})=0,h_1(\bar{x})>0,\ldots,h_l(\bar{x})>0 \}
\end{displaymath}
where $g_1,\ldots,g_k,h_1,\ldots,h_l:U\to\bbR$ are Pfaffian functions with a common chain defined on a product of open intervals $U$ in $\bbR^n$. If these functions (which we shall refer to as the functions defining $X$) have complexities at most $B$, and $k$ and $l$ are also at most $B$, then we say that the above set has complexity at most $B$.

An \emph{elementary stratum} $Y$ is an elementary semi-Pfaffian set such that, if $Y$ has codimension $m$, say, then there are, among the functions defining $Y$, some $h_1,\ldots,h_m$ vanishing identically along $Y$ such that $dh_1\wedge \cdots \wedge dh_m \ne 0$ at each point of $Y$.
\end{defn}

\begin{theorem}[\cite{GabVor-95}] \label{thm:G-Vstrat}
Let $n$ be a non-negative integer and let $B$ be a positive real number. Suppose that $U\subseteq \bbR^n$ is a product of open intervals and that $X \sset U$ is an elementary semi-Pfaffian set of complexity at most $B$. There exists a positive real number $B'$, which is bounded effectively in $B$, with the following property. There exists a partition \emph{(stratification)} of $X$ into at most $B'$ smooth (not necessarily connected) elementary strata of complexity at most $B'$, with all functions involved in their definitions having the same chain as the functions defining $X$.
\end{theorem}

We will frequently use the following result.
\begin{proposition} \label{prop:impdefzerosetbounds}
Let $n$ be a non-negative integer, let $U\subseteq \bbR^n$ be a product of open intervals and let $B$ be a positive real number. Suppose that $f:U\to \bbR$ is a function lying in the class $IP(B)$. There exists a positive real number $B'$ which is bounded effectively in $B$ such that the number of connected components of $V(f)$ is bounded by $B'$.
\end{proposition}
\begin{proof} Apply Khovanskii's Theorem \ref{thm:Khovanskii} to the system defining $f$, with an extra equation setting the coordinate corresponding to $f$ to $0$.
\end{proof}

If $f$ is unary and implicitly defined from Pfaffian functions, then this result applied to the derivative of $f$ gives an effective form of monotonicity for $f$. More formally, we have the following.
\begin{proposition} \label{prop:effmon}
Let $a$, $b$ be real numbers such that $a<b$, let $B$ be a positive real number and suppose that $f:(a,b)\to \bbR$ is a bounded function lying in the class $IP(B)$. There exists a non-negative integer $N$, and real numbers $a_{0},\ldots,a_{N+1}$ with $a=a_0<a_1<\ldots<a_N<a_{N+1}=b$, such that $N$ is bounded effectively in $B$ and the function $f$ is monotonic or constant on each interval $(a_i,a_{i+1})$, for $i=0,\ldots, N$.
\end{proposition}

\section{Decomposition Lemmas}
\label{sec:lemmas}
In this section we gather together a variety of results which will be useful in the arguments in later sections.
Our focus from now until the end of Section \ref{sec:surfaces} is on proving our effective parameterization result for surfaces implicitly defined from restricted Pfaffian functions (Theorem \ref{thm:effpara_2}). Consequently, all the results in this section will be stated in the restricted setting, for functions lying in the class $IRP$.

We begin by formalising a notion of decomposition in terms of $IRP$ functions for sets in the plane.

\begin{defn}
Let $(a,b)$ be an interval in $\bbR$, with $a \in \{-\infty\}\cup\bbR$ and $b\in \bbR\cup\{+\infty\}$, let $B$ be a positive real number and let $X$ be a subset of $\bbR^2$ of dimension $1$. Suppose that the (possibly infinite) interval $(a,b)$ is the projection of $X$ onto the first coordinate. We say that $X$ has an \emph{$IRP(B)$ decomposition} if there exist a non-negative integer $N$ and positive integers $M_1,\ldots,M_N$ which are all bounded by $B$, as well as real numbers $\eta_1,\ldots,\eta_{N}$ with $a=:\eta_0<\eta_1<\ldots<\eta_N<\eta_{N+1}:=b$, and functions $\phi_{i,j}:(\eta_i,\eta_{i+1})\to \bbR$ lying in $IRP(B)$, for $i=0,\ldots, N$ and $j=1,\ldots, M_i$, such that $\phi_{i,1}<\cdots<\phi_{i,M_i}$, for each $i = 0,\ldots,N$, and
\[
X \setminus \bigcup_{i=1}^N\left(\{\eta_i\}\times \bbR \right)= \bigcup_{i,j} \textrm{graph } (\phi_{i,j}).
\] 
\end{defn}

The following provides effective $IRP$ decompositions for the projections to the plane of bounded elementary strata of dimension $1$.
\begin{lemma}\label{lem:stratadecomp}
Let $n \geq 2$ be an integer and let $B$ be a positive real number. Suppose that $V\subseteq \bbR^n$ is a product of open intervals and that $X\subseteq V$ is a bounded elementary stratum of dimension $1$ and complexity at most $B$, whose closure is contained in $V$. Let $\tilde{\pi} \colon\bbR^n \to \bbR^2$ be the projection map down to the first two coordinates. There exists a positive real number $B'$ which is bounded effectively in $B$ such that the set $\tilde{\pi}(X)$ has an $IRP(B')$ decomposition.
\end{lemma}
\begin{proof}
As $X$ is semi-Pfaffian, the projection of $X$ down to the first coordinate is clearly a finite union of intervals and singleton sets, the number of which is bounded effectively in $B$, thus it is evidently enough to prove the statement in the case that this projection is a single interval. Since $X$ is an elementary stratum of dimension $1$, there exists a Pfaffian map $P=\langle p_1,\ldots,p_{n-1}\rangle:V\to \bbR$, whose component functions have the same chain as $X$ and complexities at most $B$, which vanish along $X$, such that at each point of $X$ we have \begin{equation}
dp_1\wedge\cdots\wedge dp_{n-1}\ne 0.\label{non-singular}
\end{equation}
Let
\[
X_0= \left\{ x\in X \; \middle| \; \det \left(\d{(P)}{(x_2,\ldots,x_n)}\right)(x)= 0\right\}.
\]
By (\ref{non-singular}) and Theorem \ref{thm:Khovanskii}, the projection of $X_0$ to the first coordinate is a finite set of points, with cardinality bounded effectively in $B$. Let $Z$ be this projection taken together with the projection of the frontier of $X$. This is a finite set, still with a bound on the cardinality that is effective in $B$. The components of \[X\setminus \bigcup_{a\in Z} \{a\} \times \bbR^{n-1}\] are the graphs of maps of the form $\langle\psi,\psi_2,\ldots,\psi_{n-1}\rangle \colon (\eta,\nu)\to \bbR^{n-1}$, where, in each case, the endpoints $\eta$ and $\nu$ lie in $Z$. 

The desired decomposition maps are the first component functions $\psi$ of each of these maps.
We just need to check that the functions $\psi$ given in this way lie in $IRP(B')$, for some positive real number $B'$ which is bounded effectively in $B$. So fix a map $\langle\psi,\psi_2,\ldots,\psi_{n-1}\rangle \colon (\eta,\nu)\to \bbR^{n-1}$ as above, for some real numbers $\eta, \nu \in Z$. We have
\begin{eqnarray*}
P(x,\psi(x),\psi_2(x),\ldots,\psi_{n-1}(x))=0, \text{ for all } x \in (\eta,\nu); \\
\det \left(\d{(P)}{(x_2,\ldots,x_n)}\right)(x,\psi(x),\psi_2(x),\ldots,\psi_{n-1}(x))\ne 0, \text{ for all } x \in (\eta,\nu).
\end{eqnarray*}
The closure of the graph of $\langle \psi,\psi_2,\ldots,\psi_{n-1}\rangle$ is contained in the closure of $X$ (and hence in $V$), as this graph is itself contained in $X$. Since $X$ is bounded, this graph is bounded. Hence, by Lemma \ref{lem:resIP}, there exists a positive real number $B'$, bounded effectively in $B$, such that the function $\psi$ lies in $IRP(B')$, as required. 
\end{proof}

The first key result of this paper concerns the effective decomposition of zero sets of certain implicitly defined functions. Here and also later we use the usual notation for cells so, for example, if $a$, $b \in \bbR \cup \{\pm\infty\}$ are such that $a < b$, and $f,g:(a,b)\to \bbR$ are continuous functions with $f(x)<g(x)$ for all $x \in (a,b)$, then we write 
\[
(f,g)_{(a,b)}= \{ \langle x,y\rangle \;|\; x \in (a,b) \text{ and } f(x)<y<g(x)\}.
\]

\begin{lemma}\label{lem:effzerosetdecomp}
Let $a$, $b \in \bbR \cup \{\pm\infty\}$ be such that $a < b$, and let $B$ be a positive real number. Suppose that $g, h:(a,b)\to \bbR$ are functions lying in $IRP(B)$ with $g<h$, and set $C$ to be the cell $(g,h)_{(a,b)}$.
Suppose further that $f: C\to \bbR$ is a function lying in $IRP(B)$ that is not identically zero.
There exists a positive real number $B'$ which is bounded effectively in $B$ such that $V(f)$ has an $IRP(B')$ decomposition.
\end{lemma}

\begin{proof}
Since $f$ lies in $IRP(B)$ there exist, by definition, a positive integer $m$, a product of open intervals $V$ in $\bbR^{2+m}$, a Pfaffian map $P=\langle p_1,\ldots, p_m\rangle:V\to \bbR$, whose component functions have complexities at most $B$, and an auxiliary map $F=\langle f, f_2,\ldots,f_m \rangle:C\to \bbR$ such that:
\begin{itemize}
\item[(i)] the closure of $\vartheta(\text{graph}(F))$ is contained in $V$;
\item[(ii)] $P(\vartheta(x,F(x)))=0$, for all $x\in C$;
\item[(iii)] $\det \left(\d{(P)}{(x_{3},\ldots,x_{n+m})}\right) (\vartheta(x,F(x)))\ne 0$, for all $x\in C$.
\end{itemize}
For simplicity we assume that $a$, $b$ are both real numbers; if $a$ is $-\infty$ (respectively $b$ is $+\infty$), use $-1$ in place of $\vartheta(a)$ (respectively $1$ in place of $\vartheta(b)$) below.
Let
\[
V^{*}(P):= \left\{x\in V(P) \; \middle| \; \det \left(\d{(P)}{(x_{3},\ldots,x_{n+m})}\right) (x)\ne 0\right\},
\]
and let $X$ be the subset given by
\[
X:= \{ x \in V^{*}(P) \;|\; x_3=0\} \cap V',
\]
where $V'$ is an open box such that $\vartheta(\text{graph}(F)) \subseteq V'$ and $\overline{V'} \subseteq V$. 
This $X$ is a bounded elementary semi-Pfaffian set whose closure is contained in $V(P)\cap \overline{V'}$, which is a subset of $V$.
Apply Theorem \ref{thm:G-Vstrat} to $X$, and let $Y$ be a stratum of the resulting decomposition that intersects $\vartheta(\text{graph}(F))$. Suppose that $Y$ has dimension $2$, and let $W$ be any connected component of $Y$ that intersects $\vartheta(\text{graph}(F))$. As $Y$ is a stratum, by the Implicit Function Theorem $W$ is the graph of a two-variable function. By analyticity, we may conclude that $f$ is identically zero, which is a contradiction. Thus $\vartheta(V(f)) \subseteq \bigcup \{ \tilde{\pi}(Y) \; | \; Y$ is a stratum of $X$ and $\dim{(Y)}\leq 1\}$, where $\tilde{\pi} :\bbR^{2+m} \to \bbR^2$ is the projection map down to the first two coordinates. Note that the number of such strata is bounded effectively in $B$.

Applying Lemma \ref{lem:stratadecomp} to each set $\tilde{\pi}(Y)$ in this union in turn, we obtain a positive real number $B''$ which is bounded effectively in $B$, and an $IRP(B'')$ decomposition of each $\tilde{\pi}(Y)$. This may then be refined to an $IRP(B''')$ decomposition of $\vartheta(V(f))$, for a positive real number $B'''$ which is bounded effectively in $B$. Each component of this decomposition is the graph of a map $\phi=\langle \phi_1,\ldots,\phi_{1+m}\rangle:(\eta,\nu)\to \bbR^{1+m}$, with $\text{graph}(\phi) \subseteq \vartheta(\text{graph}(F))$, where $(\eta,\nu)$ is a subinterval of $(\vartheta(a),\vartheta(b))$. 

Fix such a map $\phi$ and let $\psi: (\vartheta^{-1}(\eta),\vartheta^{-1}(\nu))\to \bbR$ be defined by $\psi(x)= \vartheta^{-1}(\phi_1(\vartheta(x)))$. If $x$ lies in $(\vartheta^{-1}(\eta),\vartheta^{-1}(\nu))$, then $\vartheta(x,\psi(x))=(\vartheta(x),\phi_1(\vartheta(x)))\in \tilde{\pi}(X)$, and by our choice of $\phi$ we have $f(x,\psi(x))=0$. But $\phi_1$ lies in $IRP(B''')$, and $\vartheta$ and $\vartheta^{-1}$ lie in $IRP(3)$, so there exists a positive real number $B'$ which is bounded effectively in $B$ such that $\psi$ lies in $IRP(B')$. Clearly taking all maps $\psi$ of this form provides the required $IRP(B')$ decomposition of $V(f)$.
\end{proof}

\begin{remark}\label{rmk:effzerosetdecomp}
It is clear that a transposed version of Lemma \ref{lem:effzerosetdecomp} also holds, in the following sense. Let $a, b \in \bbR \cup \{\pm\infty\}$, let $B$ be a positive real number, and let $g, h \colon (a,b)\to \bbR$ be functions lying in $IRP(B)$ with $g<h$. Define $C$ to be the cell $(g,h)_{(a,b)}$, with transpose $C^t:= \{ \langle y,x\rangle \; | \; \langle x,y\rangle \in C \}$, and suppose that $f\colon C \to \bbR$ is a function lying in $IRP(B)$. If there exist $a^{*}, b^{*} \in \bbR\cup \{\pm\infty\}$ and functions $g^{*}, h^{*} \colon (a^{*},b^{*}) \to \bbR$ lying in $IRP(B)$ such that $C^{t} = (g^{*}, h^{*})_{(a^{*},b^{*})}$ (in particular this will be true if $g$ is a constant function and $h$ is either a constant function or is monotonic decreasing), then we may apply Lemma \ref{lem:effzerosetdecomp} to the function $f^{*}\colon C^{t}\to\bbR$ given by $f^{*}(x,y)=f(y,x)$, which clearly also lies in $IRP(B)$. By transposing back the resulting $IRP$ decomposition, we again obtain a decomposition of $V(f)$ given by $IRP$ functions, this time excluding perhaps finitely many horizontal lines and given by the graphs of functions from $y$ to $x$.

In practice, we shall use not only Lemma \ref{lem:effzerosetdecomp} and the transposed version just described, but also a transposed version for functions $f \colon \widetilde{E} \to \bbR$ lying in $IRP(B)$ which have domain $\widetilde{E}$ of the form $((a,b) \times (a^{*},b^{*})) \cap C$, where $C$ is a cell of the form $(w,f)_{(a,b)}$, for $f \colon (a,b) \to (a^{*},b^{*})$ a decreasing function in $IRP(B)$ such that $\lim_{x\to b^{-}}f(x) = w$. Such a version can clearly be obtained by first decomposing using effective monotonicity (Proposition \ref{prop:effmon}) and then applying the transposed version of Lemma \ref{lem:effzerosetdecomp} outlined above.
\end{remark}

Next, we will require a result that ensures that the limits of implicitly defined functions are piecewise implicitly definable in an effective way.

\begin{lemma} \label{lem:limit}
Let $a$, $b$ be real numbers such that $0\leq a < b \leq 1$, and let $B$ be a positive real number. Suppose that $g:(a,b) \to (0,1)$ lies in the class $IRP(B)$ and set $C=(0,g)_{(a,b)}$. Suppose that $f:C\to (0,1)$ also lies in the class $IRP(B)$. Define a function $\phi:(a,b)\to [0,1]$ by
\[
\phi(x)= \lim_{y\to g(x)^-} f(x,y).
\]
There exist a positive real number $B'$ bounded effectively in $B$, a non-negative integer $N$ bounded by $B'$, and real numbers $a_1,\ldots,a_{N}$ with $a=:a_0<\ldots<a_{N+1}:=b$ such that the restriction of $\phi$ to each interval $(a_i,a_{i+1})$ lies in $IRP(B')$.
\end{lemma}

\begin{proof}
Since $f$ lies in $IRP(B)$, there exist a positive integer $m$, a product of open intervals $V$ in $\bbR^{2+m}$, a Pfaffian map $P=\langle p_1,\ldots, p_m\rangle:V\to \bbR$, whose component functions have a common chain and complexities bounded by $B$, an auxiliary map $F=\langle f_1,\ldots,f_m\rangle:C\to \bbR$ such that $f_1 = f$ and
\begin{itemize}
\item[(i)] the closure of $\vartheta(\text{graph}(F))$ is contained in $V$;
\item[(ii)] $P(\vartheta(x,y,F(x,y)))=0$, for all $x \in C$;
\item[(iii)] $\det \left(\d{(P)}{(x_{3},\ldots,x_{2+m})}\right) (\vartheta(x,y,F(x,y)))\ne 0$, for all $x \in C$.
\end{itemize}
For $i=1,\ldots,m$, define $\psi_i:(\vartheta(a),\vartheta(b))\to \bbR$ by
\[
\psi_i(x)=\lim_{y\to \vartheta(g(\vartheta^{-1}(x)))^{-}}\vartheta \left(f_i\left(\vartheta^{-1}(x), \vartheta^{-1}(y)\right)\right).
\]
Note that these limits exist and so these functions are well defined.
Put $\tilde{\psi}:= \langle\vartheta \circ g \circ \vartheta^{-1},\psi_1,\ldots,\psi_m\rangle \colon (\vartheta(a),\vartheta(b))\to \bbR^{1+m}$. We first observe that
\[
\text{graph}(\tilde{\psi}) \subseteq \overline{\vartheta(\text{graph}(F))}.
\]
It follows that the graph of $\tilde{\psi}$ lies in $V$, and moreover lies in $V(P)$.

We consider the elementary semi-Pfaffian set 
\[
V^{*}(P):=\left\{ \langle x_1,\ldots,x_{2+m} \rangle \in V(P) \;\middle|\; \det\left( \d{P}{(x_{3},\ldots,x_{2+m})}\right)(x_1,\ldots,x_{2+m})\ne 0\right\}.
\]
Define $Z:= \{x\in(\vartheta(a),\vartheta(b))\;|\; \langle x,\tilde{\psi}(x)\rangle\notin V^{*}(P)\}$ and $Z':=(\vartheta(a),\vartheta(b))\less Z$.
A priori, each of $Z$ and $Z'$ could contain a subinterval with non-empty interior.

As $\vartheta(\text{graph}(F)) \subseteq V^{*}(P)$, it follows that graph$(\tilde{\psi}\rst{Z})$ lies in the frontier of $V^{*}(P)$ taken in $V$ (i.e. lies in $(\overline{V^{*}(P)} \cap V) \setminus V^{*}(P)$; see \cite{Gab-98} for this terminology). Since $V^{*}(P)$ has dimension $2$, its frontier in $V$ has dimension at most $1$. By a theorem of Gabrielov (\cite[Theorem 1.1]{Gab-98}), it also an elementary semi-Pfaffian set, with effectively bounded complexity. 
Apply Theorem  \ref{thm:G-Vstrat} to this set and then apply Lemma \ref{lem:stratadecomp} to each of the strata $X$ of the resulting decomposition. The result is a positive real number $B'$, which is bounded effectively in $B$, and a partition of $Z$ into at most $B'$ subintervals $I$ such that, for each such $I$ with non-empty interior, $\psi_{1}\rst{I}$ lies in $IRP(B')$. 
For each such subinterval $I$, define the function $\chi\colon \vartheta^{-1}(I)\to \bbR$ by \[\chi(x) = \vartheta^{-1}(\psi_{1}(\vartheta(x))),\] for all $x \in \vartheta^{-1}(I)$. Clearly, since $\vartheta$ and $\vartheta^{-1}$ lie in $IRP(3)$, the function $\chi$ lies in $IRP(B'')$ for some $B''$ bounded effectively in $B$. Now note that
\begin{eqnarray*}
\chi(x) &=& \vartheta^{-1}\left(\lim_{y\to \vartheta(g(\vartheta^{-1}(\vartheta(x))))^{-}}\vartheta \left(f\left(\vartheta^{-1}(\vartheta(x)), \vartheta^{-1}(y)\right)\right)\right)\\
&=& \lim_{y\to \vartheta(g(x))^{-}} f(x,\vartheta^{-1}(y))  \\
&=&\lim_{y\to g(x)^{-}} f(x,y),
\end{eqnarray*}
for all $x \in \vartheta^{-1}(I)$. Therefore $\chi = \phi\rst{\vartheta^{-1}(I)}$, and hence $\phi\rst{\vartheta^{-1}(I)}$ lies in $IRP(B'')$.

We now consider the set $Z'$. Since the function $g$ lies in $IRP(B)$, there exist $m'\ge 1$ and a product of open intervals $V'$ in $\bbR^{1+m'}$, Pfaffian functions $q_1,\ldots, q_{m'}:V'\to \bbR$, which have a common chain and complexities bounded by $B$, and an auxiliary map $G=\langle g_1,\ldots,g_{m'}\rangle:(a,b)\to \bbR$ such that $g_1 = g$ and
\begin{itemize}
\item[(i)] the closure of $\vartheta(\text{graph}(G))$ is contained in $V'$;
\item[(ii)] $q_i(\vartheta(x,G(x)))=0$, for $i=1,\ldots,m'$ and for all $x \in (a,b)$;
\item[(iii)] $\det \left(\d{(q_1,\ldots,q_{m'})}{(x_{2},\ldots,x_{1+m'})}\right) (\vartheta(x,G(x)))\ne 0$, for all $x \in (a,b)$.
\end{itemize}
We let $W=(\vartheta(a),\vartheta(b))\times \pi_{m'}(V')\times \pi_{m}(V)$, where $\pi_{m'}$ and $\pi_m$ denote projections onto coordinates $\langle x_2,\ldots,x_{1+m'}\rangle$ and $\langle x_{1+m'+1},\ldots,x_{1+m'+m}\rangle$, respectively. Note that $W$ is a product of open intervals. For each $i = 1,\ldots, m'+m$ we define the function $r_i:W \to \bbR$ by
\[
r_i(t,w_1,\ldots,w_{m'},z_1,\ldots,z_m)= q_i(t,w_1,\ldots,w_{m'}),
\]
for $i=1,\ldots, m'$, and 
\[
r_{m'+j}(t,w_1,\ldots,w_{m'},z_1,\ldots,z_m)=p_j(t,w_1,z_1,\ldots,z_m),
\]
for $j=1,\ldots, m$. 
For each $i = 1,\ldots, m'+m$ we define the function $\phi_i : \vartheta^{-1}(Z') \to \bbR$ by
\[
\phi_i(x)=g_i(x),
\]
for $i=1,\ldots,m'$ and $x\in\vartheta^{-1}(Z')$, and
\[
\phi_{m'+j}(x)= \vartheta^{-1}(\psi_j(\vartheta(x)))
\]
for $j=1,\ldots,m$ and $x\in\vartheta^{-1}(Z')$.
By reasoning as above, we have that $\phi_{m'+1} = \phi\rst{\vartheta^{-1}(Z')}$. Define $\Phi:=\langle\phi_1,\ldots,\phi_{m'+m}\rangle \colon \vartheta^{-1}(Z') \to \bbR^{m'+m}$. Clearly the closure of $\vartheta(\text{graph}(\Phi))$ is contained in $W$, and $r_i(\vartheta(x),\vartheta(\Phi(x)))=0$, for all $i=1,\ldots,m'+m$ and for all $x \in \vartheta^{-1}(Z')$. It only remains to check the non-singularity condition. The Jacobian matrix has the form
$$
\begin{pmatrix}
\d{q_1}{w_1} & \d{q_1}{w_2}& \cdots & \d{q_1}{w_{m'}} & 0 &\cdots& 0 \\
\vdots &\vdots  & \vdots& \vdots & \vdots & &\vdots \\
 \d{q_{m'}}{w_1} & \d{q_{m'}}{w_2}& \cdots & \d{q_{m'}}{w_{m'}} & 0 &\cdots& 0 \\
\d{p_1}{w_1} & 0 &\cdots & 0& \d{p_1}{z_1} &\cdots &\d{p_1}{z_m} \\
\vdots &\vdots  & \vdots& \vdots & \vdots & &\vdots\\
\d{p_m}{w_1} & 0 &\cdots & 0& \d{p_m}{z_1} &\cdots & \d{p_m}{z_m} \\
\end{pmatrix}.
$$
At a point $(\vartheta(x),\vartheta(\Phi(x)))$ the upper left block has non-vanishing determinant, by the non-singularity condition satisfied by $g$. The lower right block has non-vanishing determinant, by our assumption that $\langle x,\tilde{\psi}(x)\rangle \in V^{*}(P)$ for all $x \in Z'$. So the whole matrix is non-singular, as we needed. Hence $\phi\rst{\vartheta^{-1}(Z')}$ lies in $IRP(B)$.

Combining these observations, it is straightforward to obtain the required positive real number $B'$, bounded effectively in $B$, and partition of $(a,b)$ as in the statement of the lemma.
\end{proof}

\begin{remark}
Clearly Lemma \ref{lem:limit} also applies in certain other situations. We shall apply it in the case that the domain is as above but the limit is taken as $y$ tends to $0^{+}$. We shall also apply it to certain limits in the following situation. The function $g$ is defined on a subset $(a,b)$ of $(0,1)$ and is decreasing, and the cell $C$ is taken to be $(w,g)_{(a,b)}$, where $w = \lim_{x \to b^{-}}g(x)$. The limit functions that we consider in this setting are functions of $y$, obtained by taking the limit of $f(x,y)$ either as $x$ tends to $a^{+}$ or as $x$ tends to $g^{-1}(y)^{-}$.
\end{remark}

Finally, in this section, we will need a result which allows us to `detect maximums' of implicitly defined functions.

\begin{defn}
Suppose that $a$, $b$, $a'$, $b'$ are real numbers such that $0\leq a \leq a' < b' \leq b \leq 1$,
that $g:(0,1)\times(a,b)\to \bbR$ is a continuous function and that $f:(a',b')\to \bbR$ is any function.
We say that a function $\psi:(0,f)_{(a',b')}\to (0,1)$ \emph{detects maximums of $g$} if, for each $\langle y,t\rangle \in (0,f)_{(a',b')}$, the restriction of $g(\cdot, y)$ to $[t,1-t]$ takes a maximum at $\psi(y,t)$.
\end{defn}

\begin{lemma}\label{lem:maxfn}
Let $a$, $b$ be real numbers such that $0\leq a<b\leq1$, let $B$ be a positive real number and suppose that $g:(0,1)\times(a,b)\to \bbR$ is a function lying in the class $IRP(B)$.  There exist a positive real number $B'$ bounded effectively in $B$, a non-negative integer $N$ bounded by $B'$, real numbers $a_0,\ldots,a_{N+1}$ with $a=a_0<a_1<\ldots<a_N<a_{N+1}=b$, and functions $f_i:(a_i,a_{i+1})\to (0,1)$ lying in the class $IRP(B')$, for $i=0,\ldots, N$, such that, on each cell $(0,f_i)_{(a_i,a_{i+1})}$ (with $0$ here being the constant function taking that value), there is a function in the class $IRP(B')$ which detects maximums of $g$.
\end{lemma}

\begin{proof}
\sloppy By a transposed version of Lemma \ref{lem:effzerosetdecomp} (see Remark \ref{rmk:effzerosetdecomp}), there exists a positive real number $B_1$, which is bounded effectively in $B$, such that there is a transposed $IRP(B_1)$ decomposition of $\frac{\partial g}{\partial x}$, i.e. there exist a non-negative integer $N$ and positive integers $M_0,\ldots,M_N$ which are all bounded by $B_1$, real numbers $a_0,\ldots,a_{N+1}$ with $a=a_0<a_1<\ldots<a_N<a_{N+1}=b$, and functions $\phi_{i,j}:(a_i,a_{i+1})\to (0,1)$ lying in $IRP(B_1)$, for $i=0,\ldots, N$ and $j=1,\ldots,M_i$, such that $\phi_{i,1}<\ldots<\phi_{i,M_i}$ and
\[
V\left(\frac{\partial g}{\partial x}\right)\setminus \bigcup_{i=1}^{N} \left((0,1)\times \{ a_i\}\right) = \bigcup \text{graph}^t (\phi_{i,j}),\]
where $\text{graph}^t(\phi_{i,j})= \{ \langle x,y\rangle \in (0,1)\times(a_{i},a_{i+1}) \; | \; x = \phi_{i,j}(y)\}$.

We now consider the strips $(a_i,a_{i+1})\times (0,1)$ in $\langle y,t \rangle$-space and work on each separately. So fix $i$, and from now on write $(a',b')=(a_i,a_{i+1})$ and drop the index $i$ elsewhere to make the notation clearer (so $M=M_{i}$ for example). Let
\begin{multline*}
h(y,t)= \prod_{j=1}^{M} \left( g(\phi_j(y),y)- g(t,y) \right)\prod_{j=1}^{M} \left( g(\phi_j(y),y)- g(1-t,y)\right) \cdot \\ \prod_{\substack{j_1,j_2=1 \\ j_1\ne j_2}}^{M} \left( g(\phi_{j_1}(y),y) -g(\phi_{j_2}(y),y)\right).
\end{multline*}
There is a positive real number $B_2$ bounded effectively in $B$ such that the function $h$ lies in the class $IRP(B_2)$ on the strip $(a',b')\times (0,1)$. If its zero set has interior then at least one factor is identically zero. Drop any such factors to obtain a new $h$ that is not identically zero and whose zero set has empty interior.

We now apply Lemma \ref{lem:effzerosetdecomp} to $h$, to obtain an effective $IRP$ decomposition of $V(h)$. This provides a positive real number $B_3$ which is bounded effectively in $B$, a non-negative integer $L$ and positive integers $K_0,\ldots,K_L$ all bounded by $B_3$, real numbers $\eta_0, \ldots, \eta_{L+1}$ with $a'=\eta_0<\eta_1<\ldots<\eta_{L}<\eta_{L+1}=b'$, and functions $f_{l,k}:(\eta_l,\eta_{l+1})\to (0,1)$ lying in the class $IRP(B_3)$, for $l=0,\ldots, L$ and $k= 1,\ldots,K_l$, such that $f_{l,1}<\ldots<f_{l,K_l}$ and
\[
V(h)\setminus \bigcup_{l=1}^{L} \left(\{\eta_l\}\times (0,1)\right) =\bigcup_{l,k} \text{graph}(f_{l,k}).
\]
The cells $(0,f_{l,1})_{(\eta_l,\eta_{l+1})}$ are those that we want. We now show that on each of them there is a function implicitly defined from restricted Pfaffian functions which detects maximums of $g$. We will show that on each of these cells at least one of the following functions
\begin{itemize}
\item $\langle y,t\rangle\mapsto \phi_j(y)$ for some $j=1,\ldots, M$
\item $\langle y,t\rangle\mapsto t$
\item $\langle y,t\rangle\mapsto 1-t$
\end{itemize}
detects maximums of $g$. As these all lie in some class $IRP(B_4)$, with $B_4$ a positive real number bounded effectively in $B$, this will be enough to finish the proof.

To see this, fix $C=(0,f_{l,1})_{(\eta_l,\eta_{l+1})}$ and note that, at each point $\langle y_0,t_0\rangle$ in $C$, the restriction of $g_{y_0}:=g(\cdot,
y_0)$ to $[t_0,1-t_0]$ takes a maximum at at least one of the points $\phi_1(y_0),\ldots,\phi_M(y_0),t_0,1-t_0$. So the sets
\begin{eqnarray*}
X_j =\{ \langle y,t\rangle\in C \;|\; \phi_j(y) \text{ is a point at which } g_y\rst{[t_0,1-t_0]} \text{ takes a maximum}\} \\
Y_1 =\{ \langle y,t\rangle\in C \;|\; t \text{ is a point at which } g_y\rst{[t_0,1-t_0]} \text{ takes a maximum}\} \\
Y_2 =\{ \langle y,t\rangle\in C \;|\; 1- t \text{ is a point at which } g_y\rst{[t_0,1-t_0]} \text{ takes a maximum}\}
\end{eqnarray*}
cover $C$. Suppose that two of these sets are non-empty. Then there are two of them whose closures in $C$ have non-empty intersection. Suppose $\text{cl}(X_{j_1})\cap \text{cl}(X_{j_2})\cap C$ is non-empty, with $(y_0,t_0)$ a point in the intersection. Then we have $g(\phi_{j_1}(y_0),y_0)=g(\phi_{j_2}(y_0),y_0)$. Since the $f_{l,k}$ are a decomposition of the zero set of $h$, the function $g(\phi_{j_1}(y),y)-g(\phi_{j_2}(y),y)$ must be one of the factors we omitted from $h$ for being identically zero. So $X_{j_1}=X_{j_2}$. Similarly, if $\text{cl}(X_{j})\cap \text{cl}(Y_r)\cap C$ is non-empty then $X_j=Y_r$. So one of the sets above is $C$, and the corresponding function detects maximums.
\end{proof}

\section{Effective uniform parameterization for curves}
\label{sec:curves}
In this section we begin approaching our effective parameterization result for surfaces implicitly defined from restricted Pfaffian functions, Theorem \ref{thm:effpara_2}, whose proof will be concluded in the next section. The primary result of this section is an effective uniform parameterization for certain families of one-variable functions. 

We begin by stating the formal definitions of $r$-parameterization and its analogue for functions, $r$-reparameterization. These definitions have their origins in work of Yomdin \cite{Yom-87-vge}, \cite{Yom-87-Ckr}, and Gromov \cite{Gro-87} and were given in this form by Pila and Wilkie in proving their o-minimal Reparameterization Theorem \cite{PilWil-06}.

For $m$, $l$ non-negative integers, a set $X \sset \bbR^{l}$ and a map $f = \langle f_{1},\ldots,f_{m}\rangle \colon X \to \bbR^{m}$, we use $\norm{f}$ to denote $\sup_{x\in X}\{\md{f_{1}(x)},\ldots,\md{f_{m}(x)}\}$ (where, by convention, this supremum takes value zero if $m$ is zero). Given moreover an $l$-tuple of natural numbers $\a = \langle \a_1,\ldots,\a_l\rangle$, we denote the derivative of $f$ of order $\a$ (should it exist) by \[f^{(\a)} = \left\langle \frac{\partial^{\md{\a}}f_{1}}{\partial x_1^{\a_1}\cdots\partial x_{l}^{\a_{l}}},\ldots,\frac{\partial^{\md{\a}}f_{m}}{\partial x_1^{\a_1}\cdots\partial x_{l}^{\a_{l}}}\right\rangle.\]
\begin{defn}\label{def:para}
Let $r$, $m$, $l$ be non-negative integers and let $X \sset \bbR^{m}$ be a set of dimension $l$. An \emph{$r$-parameterization} of $X$ is a finite collection of $C^{r}$ maps $\phi_{0},\ldots,\phi_{M} \colon (0,1)^{l} \to \bbR^{m}$ such that
\begin{enumerate}[(i)]
\item $X = \bigcup_{j=0}^{M}$Im$(\phi_{j})$; \label{para:cover}
\item $\norm{\phi_{j}^{(\a)}} \leq 1$, for all $j=0,\ldots,M$ and all $\a \in \bbN^{l}$ with $\md{\a}\leq r$. \label{para:bound}
\end{enumerate}
\end{defn}

\begin{defn}\label{def:repara}
Let $r$, $m$, $n$, $l$ be non-negative integers and let $f \colon X \to \bbR^{n}$ be a map whose domain $X \sset \bbR^{m}$ is a set of dimension $l$. An \emph{$r$-reparameterization} of $f$ is an $r$-parameterization $\phi_{0},\ldots,\phi_{M}$ of $X$ such that, in addition,
\begin{enumerate}[(i)]
\addtocounter{enumi}{2}
\item $f \circ \phi_{j}$ is $C^{r}$ for each $j = 0,\ldots, M$; \label{repara:Cr}
\item $\norm{(f \circ \phi_{j})^{(\a)}} \leq 1$, for all $j=0,\ldots,M$ and all $\a \in \bbN^{l}$ with $\md{\a}\leq r$. \label{repara:bound}
\end{enumerate}
\end{defn}

The proof given in the next section of Theorem \ref{thm:effpara_2}, an effective parameterization for surfaces implicitly defined from restricted Pfaffian functions, will follow the approach of \cite{PilWil-06} that a reparameterization of a certain type of two-variable map will be constructed from the reparameterizations of a suitable family of one-variable maps. In order to do this, we need to be able to reparameterize such a family in a uniform way, in the following sense. Fix non-negative integers $r$ and $n$. Consider a  family of one-variable maps $\mc{F} := \{F_{y}\colon (0,1) \to (0,1)^{n} \; | \; y \in (0,1)\}$ as a two-variable map $F\colon(0,1)^{2} \to (0,1)^{n}$ given by $F(\cdot,y) = F_{y}$. We would have a uniform way of $r$-reparameterizing the family $\mc{F}$ if there were a family of functions $\mc{S} = \{\phi_{j}\colon (0,1)^{2} \to (0,1) \; | \; j=1,\ldots,M\}$ such that the set $\mc{S}_{y}:=\{\phi_{j}(\cdot,y)\; | \; j=0,\ldots,M\}$, for each $y\in(0,1)$, were an $r$-reparameterization of the map $F(\cdot,y)$.

Given the construction that we will follow in the proof of Theorem \ref{thm:effpara_2}, we will need to be able to apply such a uniformity result to two-variable maps $F$ which are implicitly defined from restricted Pfaffian functions, and which are, in addition, defined on a wider class of cells within $(0,1)^{2}$ (described by functions which will also themselves be implicitly defined from restricted Pfaffian functions). It will also be crucial to maintain control over the complexity of the family $\mc{S}$, which will be given in terms of that of $F$ and any functions involved in defining the domain of $F$. This will, by necessity, in fact require a more precise statement concerning the construction and uniformity of the maps in $\mc{S}$ than that suggested in the previous paragraph, namely the following.

\begin{proposition} \label{prop:curvescells}
Let $n$, $r$ be non-negative integers, let $a$, $b$ be real numbers such that $0\leq a < b \leq 1$, let $B$ be a positive real number and suppose that $f \colon (a,b) \to (0,1)$ is a decreasing function lying in the class $IRP(B)$. Set $z = \lim_{x\to a^{+}}f(x)$ and $w = \lim_{x\to b^{-}}f(x)$, and let $C$ be the cell $(w,f)_{(a,b)}$ (with $w$ here denoting the constant function taking that value). Suppose that $F \colon C \to (0,1)^{n}$ is also a map lying in the class $IRP(B)$.

There exist a positive real number $B'$ which is bounded effectively in $B$, $r$ and $n$, non-negative integers $N$, $M_{0}, \ldots, M_{N}$ all bounded by $B'$, real numbers $\xi_{0}, \ldots, \xi_{N+1}$ with $w = \xi_{0} < \xi_{1} < \ldots < \xi_{N} < \xi_{N+1} = z$, and a set $\mc{S}'$ of functions
\[ \{\phi_{i,j} :C_{i} \to (0,1) \; | \; i=0,\ldots, N, j=0,\ldots, M_{i} \}, \]
where $C_{i} = (0,1)\times (\xi_i,\xi_{i+1})$, for each $i=0,\ldots,N$, with $\mc{S}' \sset IRP(B')$ such that, for each $i=0,\ldots, N$ and each $y\in (\xi_i,\xi_{i+1})$, the functions
\[ \phi_{i,0}(\cdot,y),\ldots,\phi_{i,M_{i}}(\cdot,y)\]
form an $r$-reparameterization of $F(\cdot,y)$.
\end{proposition}
\begin{proof}
Our proof follows the scheme given in \cite{Wil-15}.

First suppose that $r=1$. Write $F=\langle F_1,\ldots,F_n\rangle$ and assume that the identity function is amongst the $F_i$ (at the possible cost of increasing $n$ by $1$). For each $k, l$ with $1\le k<l\leq n$, define the function $g_{k,l}\colon C \to \bbR$ by
\[
g_{k,l} (x,y) = \left(\frac{\partial F_k}{\partial x} (x,y)\right)^2 - \left(\frac{\partial F_l}{\partial x} (x,y)\right)^2.
\]
Let $g\colon C \to \bbR$ be the product of all those $g_{k,l}$ which are not identically zero (note that this product can be implicitly defined by restricted Pfaffian functions with complexity bounded effectively in $B$ and $n$). By a transposed form of Lemma \ref{lem:effzerosetdecomp} (see Remark \ref{rmk:effzerosetdecomp}), there exists an effective transposed $IRP$ decomposition of $V(g)$, i.e. there exist a positive real number $B''$ bounded effectively in $B$ and $n$, non-negative integers $L, K_{1},\ldots,K_{L}$ all bounded by $B''$, real numbers $\eta_0,\ldots,\eta_{L+1} $ with $w =\eta_0<\eta_1<\ldots<\eta_L<\eta_{L+1} =z$ and functions $a_{i,j}:(\eta_i,\eta_{i+1}) \to (a,b)$ lying in the class $IRP(B'')$, for $i=0,\ldots, L$ and $j=1,\ldots, K_{i}$, such that $a< a_{i,1}< \ldots < a_{i,K_{i}}< b$, for all $i=0,\ldots, L$, and
\[ V(g) \setminus \bigcup_{i=1}^L \,\left((a,b) \!\times\! \{ \eta_i \}\right) = \bigcup_{\substack{i=0,\ldots, L \\ j=1,\ldots, K_{i}}} \text{graph}^t(a_{i,j}) ,\]
where $\text{graph}^t(a_{i,j}) = \{ \langle x,y\rangle \in (a,b)\!\times\!(\eta_{i},\eta_{i+1}) \; | \; x=a_{i,j}(y)\}$.

We also set $a_{i,0}(y)=a$ and $a_{i,K_{i}+1}(y)= f^{-1}(y)$, for all $i = 0,\dots,L$ and all $y\in (\eta_i,\eta_{i+1})$, the latter of which is well defined as $f$ is decreasing and analytic, hence strictly decreasing, and $f^{-1}$ is defined everywhere on $(w,z)$ and lies in the class $IRP(B)$.

For any $k,l$ with $1\le k<l\le n$, the functions
\[  \left| \frac{\partial F_k}{\partial x}(x,y) \right| -\left| \frac{\partial F_l}{\partial x}(x,y)\right|\]
have constant sign on each of the sets
\[ D_{i,j} := \{ \langle x,y\rangle \in (a,b)\!\times\! (\eta_{i},\eta_{i+1}) \;|\; x \in (a_{i,j}(y),a_{i,j+1}(y))\}, \]
for $i=0,\ldots,L$, $j=0,\ldots,K_{i}$. Therefore, for each $i=0,\ldots, L$, $j=0,\ldots,K_{i}$, there is a $k_{i,j} \in \{1,\ldots,n\}$ such that, for all $\langle x,y\rangle\in D_{i,j}$,
\[ \addtag \label{eq:Fkj>Fl}
\left| \d{F_{k_{i,j}}}{x}(x,y) \right| \ge \left| \d{F_l}{x} (x,y)\right|,
\]
for all $l=1,\ldots, n$, and, in particular,
\[ \addtag \label{eq:>1}
\left| \d{F_{k_{i,j}}}{x}(x,y) \right| \ge 1.
\]
This shows that, for a given $i\in\{0,\ldots,L\}$, either, for all $y\in (\eta_{i},\eta_{i+1})$, the function $F_{k_{i,j}}(\cdot,y)$ is strictly increasing on $(a_{i,j}(y),a_{i,j+1}(y))$, or, for all $y\in (\eta_{i},\eta_{i+1})$, the function $F_{k_{i,j}}(\cdot,y)$ is strictly decreasing on $(a_{i,j}(y),a_{i,j+1}(y))$.

\sloppy Given $i\in\{0,\ldots,L\}$, first suppose that $F_{k_{i,j}}(\cdot,y)$ is strictly increasing on $(a_{i,j}(y),a_{i,j+1}(y))$, for every $y\in (\eta_{i},\eta_{i+1})$. For each $j=0,\ldots,K_{i}$, define the functions $c_{i,j}, d_{i,j}\colon  (\eta_{i},\eta_{i+1})\to(0,1)$ as follows.

If $0< j \leq K_{i}$, define
\[
c_{i,j}(y) =F_{k_{i,j}}(a_{i,j}(y),y),
\]
so there is some positive real number $B_{i,j}$ bounded effectively in $B$ and $n$ such that $c_{i,j}$ restricted to $(\eta_{i},\eta_{i+1})$ lies in $IRP(B_{i,j})$.

If $j=0$, define
\[
c_{i,0}(y)= \lim_{x\to a^+} F_{k_{i,0}}(x,y).
\]
By Lemma \ref{lem:limit}, we can then refine $(\eta_{i},\eta_{i+1})$ into intervals on which the function $c_{i,0}$ is implicitly defined from restricted Pfaffian functions; there exists a positive real number $B_{i,0}$ bounded effectively in $B$, such that the number of intervals in the refinement is bounded by $B_{i,0}$ and the restriction of $c_{i,0}$ to each interval lies in $IRP(B_{i,0})$.

If $0\leq j< K_{i}$, define
\[
d_{i,j} (y) = F_{k_{i,j}} (a_{i,j+1}(y),y),
\]
so there is some positive real number $B'_{i,j}$ bounded effectively in $B$ and $n$ such that $d_{i,j}$ restricted to $(\eta_{i},\eta_{i+1})$ lies in $IRP(B'_{i,j})$.

Finally, if $j=K_{i}$, define
\[
d_{i,K_{i}}(y)= \lim_{x\to f^{-1}(y)^-} F_{k_{i,K_{i}}}(x,y),
\]
which is well defined as $f$ is strictly decreasing. Using Lemma \ref{lem:limit} as before, we may then refine $(\eta_{i},\eta_{i+1})$ into intervals on which $d_{i,K_{i}}$ is implicitly defined from restricted Pfaffian functions; there exists a positive real number $B'_{i,K_i}$ bounded effectively in $B$ such that the number of intervals in the refinement is bounded by $B'_{i,K_i}$ and the restriction of $d_{i,K_{i}}$ to each interval lies in $IRP(B'_{i,K_i})$.

If we instead suppose that $i\in\{0,\ldots,L\}$ is such that $F_{k_{i,j}}(\cdot,y)$ is strictly decreasing on $(a_{i,j}(y),a_{i,j+1}(y))$, for every $y\in (\eta_{i},\eta_{i+1})$, then we simply swap the definitions of $c_{i,j}$ and $d_{i,j}$, for each $j=0,\ldots,K_{i}$.

For each $i\in\{0,\ldots,L\}$, the result of this process is a positive real number $B_i$ bounded effectively in $B$ and $n$, a non-negative integer $N_{i}'$ bounded by $B_i$, and a sequence of reals $\eta_{i}=\nu_{i,0}<\nu_{i,1}<\ldots<\nu_{i,N_{i}'}<\nu_{i,N_{i}'+1}=\eta_{i+1}$ such that, on each interval $(\nu_{i,\i},\nu_{i,\i+1})$, for $\i = 0,\ldots,N'_{i}$, the functions $c_{i,0},d_{i,0},\ldots, c_{i,K_{i}},d_{i,K_{i}}$ lie in the class $IRP(B_i)$. Moreover, for each $i=0,\ldots,L$, $j = 0, \ldots, K_{i}$ and $y\in (\eta_{i},\eta_{i+1})$, the interval $(a_{i,j}(y),a_{i,j+1}(y))$ is mapped onto the interval $(c_{i,j}(y),d_{i,j}(y))$ by $F_{k_{i,j}}(\cdot,y)$. (Note that $d_{i,j}(y)=c_{i,j+1}(y)$, for all $i=0,\ldots,L$, $j = 0, \ldots, K_{i}-1$. However, we preserve this more general notation for clarity.)

Now let us temporarily fix $i\in\{0,\ldots,L\}$ and assume that we are working with $y$ in a fixed subinterval $(\nu_{i,\iota},\nu_{i,\iota+1}) \sset (\eta_{i},\eta_{i+1})$ as identified in the previous paragraph. Until otherwise stated, we now drop the index $i$, to keep the indexing manageable. Let us fix, for each $j = 0, \ldots, K$, the notation $G_{k_{j},y}$ for $F_{k_j}(\cdot, y)$, where $k_{j}$ is identified in the manner above as the index in $\{1,\ldots,n\}$ such that, for all $l = 1,\ldots,n$,
\[\md{\frac{\partial F_{k_{j}}}{\partial x}(x,y)} \geq \md{\frac{\partial F_{l}}{\partial x}(x,y)}\]
on the set $\{\langle x,y\rangle \in (a,b)\times(\nu_{\iota},\nu_{\iota+1})\; | \; x \in (a_{j}(y),a_{j+1}(y))\}$. Define, for each $j= 0, \ldots,K$, the function $\mu_{\iota,j}: (0,1)\times (\nu_{\iota},\nu_{\iota+1}) \to (0,1)$ by
\[
\mu_{\iota,j}(x,y) = G_{k_{j},y}^{-1} (c_j(y)- (d_j(y)-c_j(y))x).
\]
There exists a positive real number $B'''$ bounded effectively in $B$ and $n$ such that these functions lie in the class $IRP(B''')$.

Moreover, if we set the notation $\psi_{\iota,j,y} = \mu_{\iota,j}(\cdot,y)$, then, for each $y \in (\nu_{\i},\nu_{\i+1})$,
\[ \textrm{Im}(\psi_{\iota,j,y}) = \begin{cases} (a_{j}(y),a_{j+1}(y)) & \textrm{if } j=0,\ldots,K-1 \\ (a_{K}(y),f^{-1}(y)) & \textrm{if } j=K.\end{cases}\]
We also have that, for all $x \in (0,1)$, $y \in (\nu_{\i},\nu_{\i+1})$ and $j=0,\ldots,K$, \[ G_{k_{j},y} (\psi_{\iota,j,y}(x)) = c_{j}(y) + (d_{j}(y) - c_{j}(y))x,\]
from which it follows that
\[\left|\psi_{\iota,j,y}'(x)\right|= \left|\frac{d_{j}(y)-c_{j}(y)}{G'_{k_{j},y}(\psi_{\iota,j,y}(x))}\right| \leq 1, \]
using (\ref{eq:>1}), and, for all $l = 1,\ldots,n$,
\begin{eqnarray*} \left|(G_{l,y}\circ \psi_{\iota,j,y})'(x)\right| & = & \left| G'_{l,y}(\psi_{\iota,j,y}(x))\right|\cdot\left|\psi'_{\iota,j,y}(x)\right|\\
& = & \frac{\left|G'_{l,y}(\psi_{\iota,j,y}(x))\right|\cdot\left|d_{j}(y)-c_{j}(y)\right|}{\left|G'_{k_{j},y}(\psi_{\iota,j,y}(x))\right|}\\
& \leq & \left|d_{j}(y)-c_{j}(y)\right| \textrm{ (by (\ref{eq:Fkj>Fl}))}\\
& \leq & 1.
\end{eqnarray*}
Thus (reintroducing the parameter $i$), for each $y \in (\nu_{\iota},\nu_{\iota+1}) \sset (\eta_{i},\eta_{i+1})$, for each $i = 0, \ldots, L$ and each $\i =  0,\ldots,N'_{i}$, the set of functions $\{\psi_{\iota,0,y},\ldots,\psi_{\iota,K_{i},y}\}$ (or, in the previous notation, the set of functions $\{\mu_{\iota,0}(\cdot,y),\ldots,\mu_{\iota,K_{i}}(\cdot,y)\}$) together with the functions $\hat{a}_{i,j}(\cdot,y)$, where $\hat{a}_{i,j}(x,y) := a_{i,j}(y)$, for each $j = 1,\ldots, K_{i}$ and each $x \in (0,1)$, is a $1$-reparameterization of $F(\cdot,y)$, in a uniform sense.

Clearly, if we take $N := \sum_{i=0}^{L} N_{i}'$, $M_{i} := 2K_{i} + 1$ for each $i=1,\ldots,N$, and the real numbers $\x_{0}, \ldots, \x_{N+1}$ to be the list $\n_{0,0},\ldots,\n_{L,N_{L}'+1}$, we now have the required parameterization \[\mc{S}'_{1} = \{\phi_{i,j} :C_{i} \to (0,1) \; | \; i=0,\ldots, N, j=0,\ldots, M_{i} \} \] in the case that $r=1$.

We now continue, and prove the statement in the case that $r>1$. Still following the approach of \cite{Wil-15}, we fix an index $i \in \{0,\ldots,N\}$, and, using the terminology from the previous paragraph, we define $\widetilde{E}_{i}:=((a,b) \times (\x_{i},\x_{i+1})) \cap C$, although from now on we will again drop the index $i$ in order to make the presentation legible (we have $\widetilde{E} = \widetilde{E}_{i}$, or $M=M_i$, for example). We then define
\[\widetilde{F}\colon \widetilde{E} \to (0,1)^{(2M+1)(n+1)}\] by
\begin{multline*}
\widetilde{F}(x,y) = \langle \phi_{0}(x,y), \ldots, \phi_{M}(x,y), F_{1}(\phi_{0}(x,y)), \ldots, F_{n}(\phi_{M}(x,y)), \\ \hat{a}_{1}(x,y), \ldots, \hat{a}_{M}(x,y), F_{1}(\hat{a}_{1}(x,y)),\ldots, F_{n}(\hat{a}_{M}(x,y))\rangle.
\end{multline*}

Consider those functions $\widetilde{F}^{(q)}_{l}$, for $l=1,\ldots, (2M+1)(n+1)$ and $q = 0,\ldots,r+1$, such that $\widetilde{F}^{(q)}_{l}$ is not identically zero. There exists a positive real number $\tilde{B}$ bounded effectively in $B$, $n$ and $r$ such that the product of all these functions lies in $IRP(\tilde{B})$. By first decomposing and then applying a transposed form of Lemma \ref{lem:effzerosetdecomp} several times (see Remark \ref{rmk:effzerosetdecomp}), there exist a positive real number $\tilde{B}_{1}$ bounded effectively in $B$, $n$ and $r$, non-negative integers $L'$ and $K'_{0},\ldots,K'_{L'}$ all bounded by $\tilde{B}_{1}$, real numbers $\gamma_{0},\ldots,\gamma_{L'+1} $ with $ \x_{i} =\gamma_{0}<\gamma_{1}<\ldots<\gamma_{L'}<\gamma_{L'+1} =\x_{i+1}$, and functions $b_{s,t}:(\gamma_{s},\gamma_{s+1}) \to (a,b)$ for $s=0,\ldots, L'$ and $t=1,\ldots, K'_{s}$ lying in $IRP(\tilde{B}_{1})$ such that $a< b_{s,1}< \ldots < b_{s,K'_{s}}< b$, for all $s = 0, \ldots, L'$, the functions $a_{j}\!\upharpoonright_{(\gamma_{s},\gamma_{s+1})}$ are contained among the $b_{s,t}$, and
\[
\mathcal{Y} \setminus \bigcup_{s=1}^{L'} \left((a,b) \times \{ \gamma_{s}\}\right) = \bigcup_{\substack{s=0,\ldots, L' \\ t=1,\ldots, K'_{s}}} \text{graph}^t (b_{s,t}),
\]
where $\text{graph}^t (b_{s,t}) = \{ \langle x,y \rangle \in (a,b)\times(\x_{i},\x_{i+1}) \;| \;x=b_{s,t}(y)\}$, and $\mathcal{Y}$ is the union of the sets $V(\widetilde{F}^{(q)}_{l})$, for those $l=1,\ldots, (2M+1)(n+1)$ and $q = 0,\ldots,r+1$ for which $\widetilde{F}^{(q)}_{l}$ is not identically zero.

We also set $b_{s,0}(y)=a$ and $b_{s,K'_{s}+1}(y)= f^{-1}(y)$, for all $s = 0,\ldots, L'$ and all $y\in (\gamma_{s},\gamma_{s+1})$, the latter of which is well defined and lies in $IRP(B)$.

Then on each set $\{ \langle x,y \rangle \in \widetilde{E} \;| \;y \in (\gamma_{s},\gamma_{s+1}), x \in (b_{s,t}(y),b_{s,t+1}(y))\}$, for each $s = 0,\ldots,L'$ and each $t = 0,\ldots,K'_{s}$, each coordinate function of $\widetilde{F}$ either has no zeros or is identically zero.

Let us now define $\kappa_{s,t}\colon (0,1) \times (\gamma_{s},\gamma_{s+1}) \to (0,1)$, for $s=0,\ldots,L'$, $t=0,\ldots,K'_{s}$, by $\kappa_{s,t}(x,y) := b_{s,t}(y) + \frac{1}{2}(b_{s+1,t}(y)-b_{s,t}(y))x^{r}$. There exists a positive real number $\tilde{B}_{2}$ bounded effectively in $B$, $n$ and $r$ such that these functions lie in the class $IRP(\tilde{B}_{2})$.

Then, if we set the notation $\chi_{s,t,y} := \kappa_{s,t}(\cdot,y)$, we have $\left|\chi^{(q)}_{s,t,y}(x)\right|\leq r!$, for all $x\in(0,1)$, $y \in (\gamma_{s},\gamma_{s+1})$ and $q=0,\ldots,r$.

Moreover, if we now set $\widetilde{G}_{l,y}(x):=\widetilde{F}_{l}(x,y)$, for each $l=1,\ldots,(2M+1)(n+1)$ and $y \in (\gamma_{s},\gamma_{s+1})$, we can see, by carefully following the argument given in \cite{Wil-15} making use of the Fa\`a di Bruno formula (see, for example, \cite{KraPar-9202}), that there exists a positive real number $B^{*}$ bounded effectively in $r$ such that
\[\left|\left(\widetilde{G}_{l,y}\circ \chi_{s,t,y}\right)^{(q)}(x)\right|  \leq B^{*},\]
for all $x \in (0,1)$ and $q = 0, \ldots, r$ (this uses that
\[\md{\widetilde{G}^{(k)}_{l,y}\circ \chi_{s,t,y}} \leq \left(\frac{2k}{(b_{s+1,t}(y)-b_{s,t}(y))x^{r}}\right)^{k-1},\]
for all $k = 1, \ldots, r$; see \cite{Wil-15}, Lemma 5.8).

Now, for each $s=0,\ldots,L'$, $t=1,\ldots,K'_{s}$, set $\zeta_{s,t} \colon (0,1) \times (\gamma_{s},\gamma_{s+1}) \to (0,1)$ to be
\[\zeta_{s,t}(x,y) := b_{s,t+1}(y) - \frac{1}{2}(b_{s,t+1}(y)-b_{s,t}(y))x^{r}.\]
Likewise there exists a positive real number $\tilde{B}_{3}$ bounded effectively in $B$, $n$ and $r$ such that these functions lie in the class $IRP(\tilde{B}_3)$. We set the notation $\rho_{s,t,y}(x):=\zeta_{s,t}(\cdot,y)$ and then we similarly have both that $\left|\rho^{(q)}_{s,t,y}(x)\right|\leq r!$, and that there exists a positive real number $B^{**}$ bounded effectively in $r$ such that 
\[ \left|\left(\widetilde{G}_{l,y}\circ \rho_{s,t,y}\right)^{(q)}(x)\right| \leq B^{**},\]
for all $x \in (0,1)$, for all $l=1,\ldots,(2M+1)(n+1)$ and $q = 0, \ldots, r$.

Set $B':= \max\{B^{*},B^{**}\}$. These calculations show that, for each $s =0,\ldots,L'$, the following set $\mc{S}_{s}$ is almost the set we require on $(0,1) \times (\g_{s},\g_{s+1})$, but for the fact that we only know that the derivatives of the functions involved are bounded by $B'$ in modulus, and not necessarily by $1$.
\begin{eqnarray*}
\mc{S}_{s}& := & \{\kappa_{s,t} \; |\; t=0,\ldots,K'_{s}\} \cup \{\zeta_{s,t} \; |\; t=0,\ldots,K'_{s}\} \cup \\
& & \{\hat{b}_{s,t}: \langle x,y\rangle \mapsto b_{s,t}(y) \; | \; t=1,\ldots,K'_{s}\} \cup \\
& & \{ b^{\dagger}_{s,t}: \langle x,y\rangle \mapsto \frac{b_{s,t}(y)+b_{s,t+1}(y)}{2} \; | \; t=0,\ldots,K'_{s}\}.
\end{eqnarray*}

To conclude, therefore, we follow an argument similar to that of \cite{PilWil-06}, Corollary 5.1, but we include the details here to demonstrate that our argument is effective.

Let $D $  be the least integer greater than or equal to $B'$. Fix $s \in\{0,\ldots,L'\}$. 

For each of the functions $\t\colon (0,1) \times (\g_{s},\g_{s+1}) \to (0,1)$ lying in $\mc{S}_{s}$ and for each $k = 0,\ldots,D-1$ define the function $\t_{k}:(0,1)\times (\g_{s},\g_{s+1}) \to (0,1)$ by $\t_{k}(x,y):=\t\left(\frac{x+k}{D},y\right)$. We then set $\t_{k,y}(x):=\t_{k}(x,y)$, for each $y \in (\g_{s},\g_{s+1})$. It follows that
\[\left|\t^{(q)}_{k,y}(x)\right|, \left|(\widetilde{G}_{l,y}\circ\t_{k,y})^{(q)}(x)\right| \leq \frac{B'}{D^q} \leq \frac{B'}{D} \leq 1,\]
for all $x\in(0,1)$, and for all $l=1,\ldots,(2M+1)(n+1)$, $q=1,\ldots,r$, $\t \in \mc{S}_{s}$ and $k = 0,\ldots,D-1$. Note that, for the $q=0$ cases, we already have that $F$ and all functions $\t_{k,y}$ are bounded in modulus by $1$.

Now set $\hat{\t}_{k}(x,y):=\t\left(\frac{k}{D},y\right)$, for each $\t \in \mc{S}_{s}$ and $k = 1,\ldots,D-1$. This gives us that the following set
\begin{eqnarray*}
\mc{S}_{s}'& := & \{\t_{k} \; |\; \t \in \mc{S}_{s}, k=0,\ldots,D-1\} \\
& & \cup \; \{\hat{\t}_{k} \; |\; \t \in \mc{S}_{s}, k=1,\ldots,D-1\}
\end{eqnarray*}
consists of functions $\sg : (0,1) \times (\g_{s},\g_{s+1}) \to (0,1)$ such that the set $\{\sg(\cdot,y):(0,1)\to(0,1)\;|\;\sg \in \mc{S}'\}$ is an $r$-reparameterization of $F(\cdot,y)$, for each $y \in (\g_{s},\g_{s+1})$.

Note further that, for each $s = 0, \ldots, L'$, each $\t \in \mc{S}_{s}$ lies in the class $IRP(\tilde{B}_{4})$, for some positive real number $\tilde{B}_{4}$ which is effective in $B$, $n$ and $r$, and so there is evidently a positive real number $\tilde{B}_{5}$, effective in $B$, $n$, $r$ and $B'$, and hence in $B$, $n$ and $r$, such that, for all $s = 0, \ldots, L'$ and $\t \in \mc{S}_{s}$, the corresponding functions $\t_{k}$ and $\hat{\t}_{k}$ lie in $IRP(\tilde{B}_{5})$. Moreover, $\#\mc{S}_{s}'$ is bounded effectively in $\#\mc{S}_{s}$ and the constant $D$, and hence in $B$, $n$ and $r$.

This completes the proof in the case $r>1$, when one repeats this process for every interval $(\x_{i},\x_{i+1})$, $i=0,\ldots,N$. First take, as the new list of $\xi_{i}$s, the ordered list of all $\gamma_{i,s}$s produced, and then take the set of functions $\phi_{i,j}$ required to be the union of the sets $\mc{S}'_{s}$ given by the above method for each horizontal strip $(a,b)\times (\gamma_{i,s},\gamma_{i,s+1})$ in turn. 
The method shows that there is a bound on the number of $\xi_{i}$s and on the number of functions needed which is effective in $B$, $n$ and $r$. \end{proof}

\begin{remark}\label{rmk:curves}
It is easy to see that, by slight modifications of the proof of Lemma \ref{prop:curvescells}, we may obtain analogous results for $f$ increasing or $f$ constant on $(a,b)$. We leave the reader to formulate the appropriate statements. We may then straightforwardly combine these cases to obtain analogous statements for $f$ monotone or constant with $F$ rather defined on the cell $C := (0,f)_{(a,b)}$, where $0$ is the constant function taking that value. It is this formulation that we will apply in what follows. Also note that, as a very special case of the version of Lemma \ref{prop:curvescells} for $f$ constant, we can reparameterize a single function in $IRP$ defined on an interval in the same way.
\end{remark}

\section{Effective parameterization for surfaces}
\label{sec:surfaces}
The goal of this section is to prove Theorem \ref{thm:effpara_2}, an effective parameterization theorem for surfaces implicitly defined from restricted Pfaffian functions. We follow here the outline of the strategy of \cite{PilWil-06}, with suitable modifications made to allow us to avoid the use of ineffective tools such as the Compactness Theorem and appeals to o-minimality. As in \cite{PilWil-06}, we obtain parameterization of a surface via reparameterization of a suitable two-variable function. This reparameterization is itself obtained, as mentioned at the beginning of Section \ref{sec:curves}, from a uniform family of reparameterizations of the members of a family of one-variable functions. Roughly, this reparameterization will handle the behaviour of the function's derivatives in the first variable; in order to handle the derivatives in the second variable, we need to be able to reduce to the situation in which those derivatives are bounded. As in \cite{PilWil-06}, we consider truncations of a given function $f$ (which necessarily will have bounded derivatives), reparameterizing, letting the truncations converge to $f$, and showing that the reparameterizations of the truncated functions converge to something sufficiently close to a reparameterization of the original function $f$ that it will provide what we need.

We begin with some very useful notation followed by a helpful lemma.

\begin{notation}
Let $n$ be a non-negative integer and let $a$, $b$ be real numbers such that $0 \leq a < b \leq 1$. For maps $f: (0,1) \times (a,b) \to \bbR^{n}$ and $\phi\colon (0,1) \to Y$, for $Y \sset (a,b)$, we define $f_{\phi}\colon (0,1)^{2} \to \bbR^{n}$ by $f_{\phi}(x,y) = f(x,\ph(y))$.
\end{notation}

\begin{lemma} \label{lem:bound1subs}
Let $r$ be a non-negative integer, let $a$, $b$ be real numbers such that $0\leq a < b \leq 1$, and let $B'$ be a positive real number. Let $A$ be a subset of $\{\langle \a_{1},\a_{2}\rangle \in \bbN^{2} \; | \; \md{\a} \leq r \}$ and let $f:(0,1)\times (a,b) \to (0,1)^{n}$ be a $C^{r}$ map. Suppose that there exist a set $\mathcal{S}' \sset IRP(B')$ of at most $B'$ functions $\phi\colon (0,1)\to(a,b)$ and a finite set $X' \sset (a,b)$ of size at most $B'$ such that
\begin{enumerate}[(i)]
\item $(a,b)\setminus X' = \bigcup_{\phi \in \mc{S}'}$Im$(\phi)$;
\item $\norm{\phi^{(q)}} \leq B'$, for all $\phi \in \mc{S}'$ and all $q=0,\ldots,r$;
\item $\norm{(f _\phi)^{(\a)}} \leq 1$, for all $\phi \in \mc{S}'$ and all $\a \in A$ with $\a_{2} = 0$;
\item $\norm{(f _\phi)^{(\a)}} \leq B'$, for all $\phi \in \mc{S}'$ and all $\a \in A$ with $\a_{2} > 0$.
\end{enumerate}
There exist a positive real number $B$, which is bounded effectively in $B'$, a set $\mc{S}\sset IRP(B)$, and a finite set $X\sset(a,b)$, such that $\mc{S}$ is an $r$-parameterization of $(a,b)\setminus X$, the cardinalities of both $\mc{S}$ and $X$ are bounded by $B$, and, for all $\psi \in \mc{S}$, $\norm{(f_{\psi})^{(\a)}} \leq 1$, for all $\a \in A$.
\end{lemma}
\begin{proof}
Let $X'\sset (a,b)$ and $\mc{S}'$ be as in the hypotheses of the lemma. The construction of $\mc{S}$ follows a similar argument to the final step in the proof of Lemma \ref{prop:curvescells}. Let $D$ be the least integer greater than or equal to $B'$. For each $\phi \in \mc{S}'$ and each $k=0,\ldots,D-1$, define $\phi_{k}\colon (0,1) \to (0,1)$ by $\phi_{k} = \phi\circ\lambda_{k}$, where $\lambda_{k}\colon (0,1) \to (0,1)$ is the linear function given by $\lambda_{k}(x) = \frac{x+k}{D}$.

Set $\mc{S}:=\{\phi_{k} \; | \; \phi \in \mc{S}', k=0,\ldots,D-1 \}$, and $X := X' \cup \{\frac{1}{D},\ldots,\frac{D-1}{D}\}$. Clearly $\mc{S}$ is an $r$-parameterization of $(a,b)\less X$.

Fix $k\in \{0,\ldots,D-1\}$ and $\a \in A$. If $\a_{2} = 0$, then we have $(f_{\phi_{k}})^{(\a)} = (f_{\phi})^{(\a)}$ and, by assumption, $\norm{(f_{\phi})^{(\a)}} \leq 1$, so we are done. If $\a_{2} > 0$, then, using $(f_{\phi_{k}})^{(\a)} = \left(\left(f_{\phi}\right)_{\lambda_{k}}\right)^{(\a)} = \frac{\left(\left(f_{\phi}\right)^{(\a)}\right)_{\lambda_{k}}}{D^{\a_{2}}}$, we have that
\[
\norm{(f_{\phi_{k}})^{(\a)}} \leq \frac{\norm{\left(f_{\phi}\right)^{(\a)}}}{D^{\a_{2}}} \leq \frac{B'}{D^{\a_{2}}} \leq \frac{B'}{D} \leq 1.
\]
Finally note that, by effective choice of $D$ in terms of $B'$, and by the construction of the functions in $\mc{S}$, there clearly exists a positive real number $B$, effective in $B'$, such that $\mc{S} \sset IRP(B)$ and $\#\mc{S}$ is bounded by $B$.
\end{proof}

We come now to the first result containing the key idea described in the introduction to this section, namely the use of reparameterization to reduce from a situation in which the derivatives of a two-variable function with respect to the first variable are bounded to one in which derivatives with respect to both variables are bounded. For the analogous result in the context of the Pila--Wilkie Theorem we refer to \cite[Lemma 4.3]{PilWil-06}, but our proof of the following statement demonstrates that effective bounds can be obtained and that we can remain within the setting of functions implicitly defined from restricted Pfaffian functions.

\begin{lemma}\label{lem:PW4.3}
Let $a$, $b$ be real numbers such that $0\leq a < b \leq 1$, and let $B$ be a positive real number. Suppose that $f:(0,1)\times (a,b) \to (0,1)$ is a function lying in the class $IRP(B)$. Suppose further that, for all $\langle x,y \rangle \in (0,1)\times (a,b)$, we have
\[
\md{\frac{\partial f}{\partial x}(x,y)} \leq 1.
\]
For each integer $r\geq 2$, there exist a positive real number $B'$ bounded effectively in $B$ and $r$, a finite set $X \sset (a,b)$ and an $(r-1)$-parameterization $\mathcal{S}$ of the cofinite set $(a,b)\setminus X$ such that $\mathcal{S} \sset IRP(B')$, the cardinalities of $\mc{S}$ and of the finite set $X$ are at most $B'$ and, for each $\phi \in \mathcal{S}$, the function $f_{\phi}$ has both of its first-order partial derivatives bounded by $1$.
\end{lemma}

\begin{proof}
We begin by applying Lemma \ref{lem:maxfn} to the function $g \colon (0,1) \times (a,b) \to \bbR$ defined by
\[\langle x,y\rangle \mapsto \left(\frac{\partial f}{\partial y}(x,y)\right)^2.\]
This gives us the existence of a positive real number $B_1$ bounded effectively in $B$, a non-negative integer $L$ bounded by $B_1$, real numbers $a_0, \ldots, a_{L+1}$ with $a=a_0<a_1<\ldots<a_L<a_{L+1}=b$, as well as functions $f_l:(a_l,a_{l+1})\to (0,1)$ lying in the class $IRP(B_1)$ such that, on each cell $C_{l}:=(0,f_l)_{(a_l,a_{l+1})}$ (with $0$ the constant function), there is a function $s_{l} \colon C_{l} \to (0,1)$ lying in the class $IRP(B_1)$ which detects maximums of $g$. By Proposition \ref{prop:effmon}, we may also assume that that each $f_l$ is either strictly monotonic or constant (increasing $B_1$ if necessary by an effective amount).

For each $l = 0, \ldots, L$, define the map $h_{l}\colon C_{l} \to (0,1)^{2}$ by \[h_{l}(y,t) = \langle s_{l}(y,t), f(s_{l}(y,t),y) \rangle.\] Appealing to Remark \ref{rmk:curves}, we may apply an improved version of Lemma \ref{prop:curvescells} to $h_{l}$, for each $l = 0, \ldots, L$ in turn, to obtain a positive real number $\xi$, a positive real number $B_2$ bounded effectively in $B$ and $r$, and, for each $l=0,\ldots, L$, a non-negative integer $M_{l}$ bounded by $B_2$ and functions $\phi_{l,0},\ldots,\phi_{l,M_{l}}:(0,1)\times (0,\xi) \to (a_l,a_{l+1})$ lying in $IRP(B_2)$ such that, for each positive real number $t<\xi$, the functions $\phi_{l,0}(\cdot,t),\ldots, \phi_{l,M_{l}}(\cdot, t)$ form an $r$-reparameterization of \[h_{l}(\cdot,t) \colon \{y \in (a_l,a_{l+1}) \; | \; \langle y,t\rangle \in C_{l}\} \to (0,1)^{2}.\] (Note that Lemma \ref{prop:curvescells} also provides us with parameterizations further up the $t$-axis, but we do not need these.)

Fix $l \in \{0, \ldots, L\}$. Define, for each $j = 0,\ldots,M_{l}$, the pointwise limit function $\mu_{l,j}\colon (0,1) \to [a_l,a_{l+1}]$ given by $\mu_{l,j}(y)=\lim_{t\to 0^{+}}\phi_{l,j}(y, t)$. By Lemma \ref{lem:limit}, there exist a positive integer $B_3$ bounded effectively in $B$ and $r$ and, for each $j = 0, \ldots, M_{l}$, a positive integer $N_{l,j}$ bounded by $B_3$ and real numbers $b_{l,j,0},\ldots,b_{l,j,N_{j}+1}$ with $0=b_{l,j,0}<b_{l,j,1}<\ldots<b_{l,j,N_{j}}<b_{l,j,N_{j}+1}=1$ such that, on each interval $(b_{l,j,i},b_{l,j,i+1})$, the restriction of $\mu_{l,j}$ lies in $IRP(B_3)$. By subdividing further in an effective way, using Proposition \ref{prop:effmon}, we may also assume that the restriction of $\mu_{l,j}$ to each interval $(b_{l,j,i},b_{l,j,i+1})$ is monotonic or constant.

Now, for each $l=0,\ldots,L$, $j = 0,\ldots,M_{l}$ and $i=0,\ldots,N_{l,j}$, define $\psi_{l,j,i} \colon (0,1) \to [a_{l},a_{l+1}]$ to be the function $\psi_{l,j,i}(y) = \mu_{l,j}((b_{l,j,i+1} - b_{l,j,i})y + b_{l,j,i})$. Then each $\psi_{l,j,i}$ is monotonic or constant, and there exists a positive real number $B_4$ bounded effectively in $B$ and $r$ such that each $\psi_{l,j,i}$ lies in $IRP(B_4)$. We set, for each $l=0,\ldots,L$,
\[\mc{S}_{l,0} := \{ \psi_{l,j,i}\;|\; j = 0,\ldots,M_{l}, i = 0,\ldots,N_{l,j}\textrm{ and } \ex y\,\in(0,1) \;\psi_{l,j,i}(y)\notin \{a_{l},a_{l+1}\} \}.\]
The union of these sets as $l$ varies will be almost the $(r-1)$-reparameterization that we require. Note that the functions in $\mc{S}_{l,0}$ are $C^{r-1}$ (indeed are they are analytic). Moreover, by Corollary 3.7 of \cite{Tho-12} (see also Remark 4.1 of \cite{PilWil-06}), they have derivatives up to order $r-1$ bounded by $1$. In addition, there exists a finite set $X_{l} \sset (a_l,a_{l+1})$ such that the functions in $\mc{S}_{l,0}$ cover $(a_l,a_{l+1}) \less X_{l}$. Therefore $\mc{S}_{l,0}$ is in fact an $(r-1)$-parameterization of $(a_l,a_{l+1}) \less X_{l}$. Moreover, the size of the finite set $X_{l}$ is bounded effectively in $M_{l}$ and the $N_{l,j}$, for $j=0,\dots,M_{l}$, hence is bounded effectively in $B$ and $r$.

It only remains to consider the first-order partial derivatives of $f_{\ps}$, for $\ps \in \mc{S}_{l,0}$, $l = 0,\ldots,L$. By the lemma hypothesis, it is immediate that $\md{\frac{\partial f_\ps}{\partial x}(x,y)} \leq 1$, for all $\langle x,y \rangle \in(0,1)^{2}$. As for bounding $\frac{\partial f_\ps}{\partial y}(x,y)$, note that, for each $\ps \in \mc{S}_{l,0}$, there exists $j \in \{0,\ldots,M_{l}\}$ and a linear function $\lmb \colon (0,1) \to (0,1)$ such that $\ps = \mu_{l,j} \circ \lmb$. Consequently
\begin{align*}
\md{\frac{\partial f_\ps}{\partial y}(x,y)} & \leq \md{\frac{\partial f_{\mu_{l,j}}}{\partial y}(x,y)},
\end{align*}
for such $j\in \{0,\ldots,M_{l}\}$ and for all $\langle x,y\rangle \in (0,1)^{2}$.
Therefore, we will consider bounding $\frac{\partial f_{\mu_{l,j}}}{\partial y}(x,y)$, for $\mu_{l,j}$ with $l = 0,\ldots, L$, $j=0,\ldots,M_{l}$, and $\langle x,y\rangle \in (0,1)^2$.

Fix $\langle x_{0},y_{0}\rangle\in (0,1)^{2}$, $l\in \{0,\ldots,L\}$ and $j \in \{0,\ldots, M_{l}\}$. Recall that, by definition, $\mu_{l,j} (y) = \lim_{t\to 0^{+}}\phi_{l,j}(y, t)$. As $r \geq 2$, we also have that $\mu'_{l,j}(y) = \lim_{t\to 0^{+}}\frac{\partial \phi_{l,j}}{\partial y}(y, t)$. Therefore, for sufficiently small $t \in (0, \x)$, both
\[\md{\frac{\partial f}{\partial y}(x_{0},\mu_{l,j}(y_{0})) - \frac{\partial f}{\partial y}(x_{0},\phi_{l,j}(y_{0},t))} \leq 1,\]
by continuity of $\frac{\partial f}{\partial y}$, and
\[\md{\mu'_{l,j}(y_{0}) - \frac{\partial \ph_{l,j}}{\partial y}(y_{0},t)} \leq \frac{1}{\md{\frac{\partial f}{\partial y}(x,\mu_{l,j}(y_{0}))}}.\]
Fix $t \in (0, \x)$ sufficiently small such that both of the previous two inequalities hold. We then have that

\begin{align}
\md{\frac{\partial f_{\mu_{l,j}}}{\partial y}(x_{0},y_{0})} & = \md{\mu'_{l,j}(y_{0})\cdot\frac{\partial f}{\partial y}(x_{0},\mu_{l,j}(y_{0}))} \nonumber\\
& \leq  \md{\frac{\partial \ph_{l,j}}{\partial y}(y_{0},t)}\cdot\md{\frac{\partial f}{\partial y}(x_{0},\mu_{l,j}(y_{0}))} + 1 \nonumber\\
& \leq  \md{\frac{\partial \ph_{l,j}}{\partial y}(y_{0},t)}\cdot\md{\frac{\partial f}{\partial y}(x_{0},\phi_{l,j}(y_{0},t))} + \md{\frac{\partial \ph_{l,j}}{\partial y}(y_{0},t)} + 1 \label{eq:deriv_fphi}.
\end{align}
\sloppy Recall that the functions $\phi_{l,0}(\cdot,t),\ldots, \phi_{l,M_{l}}(\cdot, t)$ form an $r$-reparameterization of $h_{l}(\cdot,t) = \langle s_{l}(\cdot,t), f(s_{l}(\cdot,t),\cdot) \rangle$. Therefore, we have that $\phi_{l,j}(y_{0},t) \in C_l$, and furthermore that
\begin{align}
\md{\frac{\partial}{\partial y}(s_{l}(\phi_{l,j}(y,t),t))(y_{0},t)} & \leq 1 \label{eq:derivs_l}\\
 \textrm{and } \md{\frac{\partial}{\partial y}(f(s_{l}(\phi_{l,j}(y,t),t),\phi_{l,j}(y,t)))(y_{0},t)} & \leq 1.\notag
\end{align}
The second of these inequalities gives us that
\begin{align*}
\begin{split}
\left|\frac{\partial}{\partial y}(s_{l}(\phi_{l,j}(y,t),t))(y_{0},t)\cdot\frac{\partial f}{\partial x}(s_{l}(\phi_{l,j}(y_{0},t),t),\phi_{l,j}(y_{0},t)) + \right.\\
\left. \frac{\partial \ph_{l,j}}{\partial y}(y_{0},t)\cdot\frac{\partial f}{\partial y}(s_{l}(\phi_{l,j}(y_{0},t),t),\phi_{l,j}(y_{0},t)) \right| & \leq 1.
\end{split}
\end{align*}
Combining this with (\ref{eq:derivs_l}) and with the lemma hypothesis, we see that
\begin{align*}
\md{\frac{\partial \ph_{l,j}}{\partial y}(y_{0},t)}\cdot\md{\frac{\partial f}{\partial y}(s_{l}(\phi_{l,j}(y_{0},t),t),\phi_{l,j}(y_{0},t))} \leq 2.
\end{align*}
Now note that it follows from the definition of $s_{l}$ that
\begin{align*}
\md{\frac{\partial f}{\partial y}(s_{l}(y,t),y)} & \geq \md{\frac{\partial f}{\partial y}(x,y)},
\end{align*}
for all $x \in [t,1-t]$, for all $\langle y,t\rangle \in C_{l}$. Therefore, using the fact that $\phi_{l,j}(y_{0},t) \in C_{l}$, it follows from (\ref{eq:deriv_fphi}) that

\begin{align*}
\md{\frac{\partial f_{\zeta_{l,j}}}{\partial y}(x_{0},y_{0})} & \leq \md{\frac{\partial \ph_{l,j}}{\partial y}(y_{0},t)}\cdot\md{\frac{\partial f}{\partial y}(s_{l}(\phi_{l,j}(y_{0},t),t),\phi_{l,j}(y_{0},t))} + \md{\frac{\partial \ph_{l,j}}{\partial y}(y_{0},t)} + 1\\
& \leq  2 + 1 + 1 \\
& = 4.
\end{align*}
Since the bound on $\norm{\frac{\partial f_{\mu_{l,j}}}{\partial y}}$, and hence on $\norm{\frac{\partial f_{\psi}}{\partial y}}$, is therefore out by only a factor of a positive absolute constant from the bound of $1$ that we require, we may obtain the required reparameterization by applying Lemma \ref{lem:bound1subs} to $\mc{S}':= \Un_{l=0}^{L}\mc{S}_{l,0}$ together with $X':= \Un_{l=0}^{L}\{a_{l}\} \un \Un_{l=0}^{L} X_{l}$. This gives us a finite set $X \sset (a,b)$ containing $X$ and an $(r-1)$-parameterization $\mc{S}$ of $(a,b)\less X$ with the required properties.
\end{proof}

From this we obtain the following corollary, the analogue of the previous result for two-variable maps into $(0,1)^{n}$.

\begin{corr} \label{cor:PW4.3}
Let $n$ be a non-negative integer, let $a$, $b$ be real numbers such that $0\leq a < b \leq 1$ and let $B$ be a positive real number. Suppose that $f = \langle f_{1},\ldots,f_{n}\rangle:(0,1)\times (a,b) \to (0,1)^{n}$ lies in the class $IRP(B)$. Suppose further that, for all $\langle x,y \rangle \in (0,1)\times (a,b)$, we have
\[
\md{\frac{\partial f_{j}}{\partial x}(x,y)} \leq 1,
\]
for all $j=1,\ldots,n$. For each integer $r\geq 2$, there exist a positive real number $B'$ bounded effectively in $B$, $r$ and $n$, a finite set $X \sset (a,b)$ and an $(r-1)$-parameterization $\mathcal{S}$ of the cofinite set $(a,b)\setminus X$ such that $\mathcal{S} \sset IRP(B')$, the cardinalities of $\mc{S}$ and of the finite set $X$ are bounded by $B'$ and, for each $\phi \in \mathcal{S}$, the map $f_{\phi}$ is $C^{1}$ and has both of its first-order partial derivatives bounded by $1$.
\end{corr}
\begin{proof}
The proof follows a suggestion to be found in \cite{PilWil-06}, namely that the proof of a result of this kind should follow a similar argument to that of \cite{PilWil-06}, Lemma 3.5. However, we include the details as we require that certain derivatives are not only bounded, but are bounded by $1$, and to demonstrate that effective bounds can be obtained.

We therefore proceed via induction on $n$, with Lemma \ref{lem:PW4.3} as the base case. For $n>1$, let $f\colon (0,1)\times(a,b)\to (0,1)^{n}$ be as in the statement and define $F\colon (0,1)\times (a,b) \to (0,1)^{n-1}$ by $F(x,y) = \langle f_{1}(x,y),\ldots,f_{n-1}(x,y)\rangle$. Applying the inductive hypothesis to $F$ gives us a positive real number $B'_{F}$ bounded effectively in $B$, $n$ and $r$, a finite set $X_{F} \sset (a,b)$ and an $(r-1)$-parameterization $\mathcal{S}_{F}$ of $(a,b)\less X_{F}$ such that $\mathcal{S}_{F} \sset IRP(B'_{F})$, we have $\#\mathcal{S}_{F}$, $\# X_{F} \leq B'_{F}$ and, for each $\phi \in \mathcal{S}_{F}$, the function $F_{\ph}$ is $C^1$ and has both of its first-order partial derivatives bounded by $1$.

For each $\phi \in \mc{S}_{F}$, likewise apply Lemma \ref{lem:PW4.3} to the function $(f_{n})_{\ph}$ to obtain a positive real number $B'_{\ph}$ bounded effectively in $B$ and $r$, a finite set $X_{\phi} \sset (0,1)$ and an $(r-1)$-parameterization $\mathcal{S}_{\phi}$ of $(0,1)\less X_{\phi}$ such that $\mathcal{S}_{\ph} \sset IRP(B'_{\ph})$, we have $\#\mathcal{S}_{\ph}$, $\# X_{\ph} \leq B'_{\ph}$ and, for each each $\psi \in \mathcal{S}_{\ph}$, the function $(f_{n})_{\ph\circ\ps}$ is $C^1$ and has both of its first-order partial derivatives bounded by $1$.

Now consider the set $\mc{S}:= \{\ph\circ\ps\colon(0,1) \to (a,b) \;|\; \phi \in \mathcal{S}_{F} \textrm{ and } \ps \in \mathcal{S}_{\phi}\}$. Clearly $\#\mc{S}$ is bounded effectively in terms of $B'_{F}$ and the $B'_{\ph}$, for $\phi \in \mathcal{S}_{F}$. Moreover, the functions in this set are $C^{r-1}$ (indeed they are analytic). After a further round of linear substitutions, we may further assume that they have derivatives up to order $r-1$ which are bounded by $1$. In addition, they are implicitly defined from restricted Pfaffian functions with complexity of implicit definition bounded effectively in $B'_{F}$ and in the corresponding $B'_{\ph}$, and they also cover a set of the form $(a,b)\less X$, where $X$ is a finite subset of $(a,b)$ with $\# X \leq \# X_{F} + \Sigma_{\ph \in \mc{S}_{F}}\# X_{\ph}$, i.e. $\# X$ is bounded effectively in $B'_{F}$ and in the $B'_{\ph}$, for $\phi \in \mathcal{S}_{F}$. Taking the largest of these bounds gives us a positive real number $B'$ which bounds all three that is clearly bounded effectively in $B$, $r$ and $n$. Furthermore, it is easy to see that, for each $\ph\circ\ps \in \mc{S}$ and $i=1,\ldots,n$, the function $(f_{i})_{\ph\circ\ps}$ is in $C^1$ and, via the Chain Rule, has each of its first-order partial derivatives bounded by $1$, as required.
\end{proof}

Now we come to our version of \cite{PilWil-06}, Lemma 4.4, which extends the idea of Corollary \ref{cor:PW4.3} to higher order derivatives.

\begin{lemma} \label{lem:PW4.4}
Let $n$, $r$ be non-negative integers, let $a$, $b$ be real numbers such that $0\leq a < b \leq 1$ and let $B$ be a positive real number. Suppose that $f\colon(0,1)\times (a,b)\to (0,1)^{n}$ lies in the class $IRP(B)$. Suppose further that, for all $\langle x,y \rangle\in(0,1)\times (a,b)$,
\[\left| \frac{\partial^{i} f_{j}}{\partial x^{i}}(x,y)\right|\leq 1,\]
for all $j=1,\ldots,n$ and all $i= 0, \ldots, r$.

For each non-negative integer $k$, there exist a positive real number $B_{k}$ bounded effectively in $B$, $r$, $k$ and $n$, a finite set $X_{k}\subseteq (a,b)$ and an $r$-parameterization $\mathcal{S}_{k}$ of the cofinite set $(a,b)\setminus X_{k}$ such that $\mathcal{S}_{k} \sset IRP(B_k)$, the cardinalities of $\mc{S}_{k}$ and of the finite set $X_k$ are bounded by $B_k$ and, for each $\phi \in \mc{S}_k$, the map $f_{\phi}$ is $C^{r}$ and, for each $\alpha = \langle \a_{1},\a_{2}\rangle \in \bbN^{2}$ with $\md{\alpha}\leq r$ and $\alpha_{2}\leq k$, we have $\norm{(f_\phi)^{(\alpha)}} \leq 1$.

\end{lemma}
\begin{proof}
We prove this by induction on $k$. For $k=0$ we can take $X_0$ to be the empty set and $\mc{S}_0$ to consist solely of the identity function on $(0,1)$.

Now we suppose that $\mc{S}_{k}$ and $X_{k}$ have been constructed. Let \[\Delta:= \{ \alpha=\langle \a_{1},\a_{2}\rangle \in \bbN^{2}\;|\; \left|\alpha\right|\leq r-1,\alpha_2\leq k\}\] and set $\tilde{n}=\# \Delta\cdot \#\mathcal{S}_{k}$. Let $F=\langle F_1,\ldots, F_{\tilde{n}}\rangle:(0,1)^2 \to \bbR^{\tilde{n} \cdot n}$ be a map whose component functions form an enumeration of all component functions of the maps $\left( f_\phi\right)^{(\alpha)}: (0,1)^2\to \bbR^n$, for $\ph \in \mc{S}_{k}$ and $\a \in \Delta$. Then the hypotheses of Corollary \ref{cor:PW4.3} hold for $F$, with $r+1$ in place of $r$. Applying this result, we obtain a positive real number $B'_{k+1}$ which is bounded effectively in $B$, $r$, $k$ and $n$, a finite set $Y_{k+1}\subseteq (0,1)$ and an $r$-parameterization $\mathcal{S}$ of $(0,1)\setminus Y_{k+1}$ such that $\mc{S} \sset IRP(B'_{k+1})$, we have $\#Y_{k+1}$, $\#\mc{S} \leq B'_{k+1}$ and, for each $\psi \in \mathcal{S}$ and each $i=1,\ldots,\tilde{n}$, the function $(F_i)_\psi$ is $C^1$ and has both first-order partial derivatives bounded by $1$. That is, for each $\phi\in \mathcal{S}_{k}$, each $\psi \in \mathcal{S}$ and each $\alpha \in \Delta$ we have
\begin{equation}\label{our4-4inductive}
\norm{\frac{\partial}{\partial x} \left(\left( \left( f_\phi\right)^{(\alpha)}\right)_\psi\right)}, \norm{\frac{\partial}{\partial y} \left(\left( \left( f_\phi\right)^{(\alpha)}\right)_\psi\right)} \leq 1.
\end{equation}

Let
\[
\mathcal{S}'_{k+1} = \{ \phi\circ\psi \colon (0,1) \to (a,b)\less X_{k} \; |\; \phi\in \mathcal{S}_k,\psi\in\mathcal{S}\}.
\]
Note that the functions in $\mc{S}'_{k+1}$ are implicitly defined from restricted Pfaffian functions with the complexities of their implicit definitions bounded effectively in $B, r, k$ and $n$, and moreover they are clearly $C^{r}$ (indeed they are analytic). Moreover, there is a positive real number $B'$ which is effective in $r$ such that $\norm{\left(\phi \circ \psi\right)^{(q)}} \leq B'$, for all $q=0,\ldots,r$. There is also a finite set $X'_{k+1} \sset (a,b)$ such that the functions in $\mc{S}_{k+1}'$ cover $(a,b) \less X'_{k+1}$. Furthermore, $\# X'_{k+1}$ is bounded effectively in $\# X_{k}$ and $\#Y_{k+1}$, hence by $B$, $r$, $k$ and $n$. We will show that the set $\mc{S}'_{k+1}$ is almost the set $\mc{S}_{k+1}$ that we require.

Let $\alpha = \langle \a_{1}, \a_{2}\rangle \in \bbN^2$, with $\left|\alpha\right|\le r,\alpha_{2}\le k+1$. Note that if $\alpha_2=0$, then $\left(f_{\phi\circ\psi}\right)^{(\alpha)}= \left(\left( f_\phi\right)^{(\alpha)}\right)_\psi$ and so $\left| \left|\left(f_{\phi\circ\psi}\right)^{(\alpha)}\right|\right|\leq \left|\left|\left( f_\phi\right)^{(\alpha)}\right|\right|\leq 1$, by the hypothesis of the lemma.

Now suppose that $\a_{2}>0$. We claim that there is a positive integer $B''$ depending only on $\alpha$, and effectively computable from $\alpha$, such that if $\phi\circ\psi\in \mathcal{S}'$ then $\left| \left|\left(f_{\phi\circ\psi}\right)^{(\alpha)}\right|\right|\leq B''$ on $(0,1)^2$. 
In this case $\left( f_{\phi\circ\psi}\right)^{(\alpha)}=\frac{\partial}{\partial y}\left(\left(f_{\phi\circ\psi}\right)^{(\beta)}\right)$, where $\beta=\langle \alpha_1,\alpha_2-1\rangle $ is in $\Delta$. A calculation (for example using the Fa\`a di Bruno formula) shows that, for $\langle x_0, y_0 \rangle\in(0,1)^2$, we have
\[
\left(f_{\phi\circ\psi}\right)^{(\beta)}(x,y) =P\left(\left\{ \left(\left(f_\phi\right)^{(\gamma)}\right)_\psi (x,y)\right\}_{ \left|\gamma\right|\le \left|\beta\right|},\left\{ \psi^{(j)}(y)\right\}_{j \le \alpha_2-1} \right)
\]
where $P$ is a polynomial in the data shown, with $P$ depending only on $\beta$. Differentiating both sides with respect to $y$ we have
\[
\left( f_{\phi\circ\psi}\right)^{(\alpha)} (x,y) = Q \left(\left\{ \frac{\partial}{\partial y} \left(\left( f_\phi\right)^{(\gamma)} \right)_{\psi}(x,y)\right\}_{\left|\gamma\right|\leq \left|\beta\right|},\left\{ \psi^{(j)}(y) \right\}_{j\le \alpha_2}\right)
\]
for  $\langle x_0, y_0 \rangle\in(0,1)^2$, with $Q$ a polynomial depending only on $\alpha$. Since $\psi\in\mathcal{S}$, the derivatives of $\psi$ shown are bounded by $1$ in modulus. Moreover, by \eqref{our4-4inductive} the derivatives of $f_\phi$ are also bounded by $1$ in modulus. So $\norm{\left(f_{\phi\circ\psi}\right)^{(\alpha)}} \le B''$, for some $B''$ depending only on $Q$ and so only on $\alpha$, as claimed.

Therefore, we may finish by applying Lemma \ref{lem:bound1subs} to $\mc{S}'_{k+1}$ together with $X'_{k+1}$ to obtain the required $B_{k+1}$, $\mc{S}_{k+1}$ and $X_{k+1}$.
\end{proof}

With these lemmas in place, we come to our effective parameterization and reparameterization results, the remaining statements and proofs in this section. These are our analogues to the proofs given in Section 5 of \cite{PilWil-06}.

Our first result of this kind is for cells lying in $(0,1)^{2}$ defined by functions which are implicitly defined by restricted Pfaffian functions.

\begin{theorem} \label{thm:effpara_cell2}
Let $r$ be a non-negative integer, let $a$, $b$ be real numbers such that $0 \leq a < b \leq 1$ and let $B$ be a positive real number. Suppose that $g, h\colon (a,b) \to (0,1)$ are functions lying in the class $IRP(B)$ with $g<h$. There exist a positive real number $B'$, which is bounded effectively in $B$ and $r$, and an $r$-parameterization $\mc{S}$ of the cell $(g,h)_{(a,b)}$ such that $\mc{S}\sset IRP(B')$ and the cardinality of $\mc{S}$ is bounded by $B'$.
\end{theorem}
\begin{proof}
The proof corresponds to part of the proof of (II)$_{1}$ in \cite{PilWil-06}, Section 5. Applying the special case of Proposition \ref{prop:curvescells} (mentioned in Remark \ref{rmk:curves}) to the map $\langle g,h \rangle \colon (a,b) \to (0,1)^{2}$, we obtain a positive real number $B''$ which is bounded effectively in $B$ and $r$, and an $r$-reparameterization $\mc{S}'$ of $\langle g, h \rangle$ such that $\mc{S}' \sset IRP(B'')$ and $\#\mc{S}' \leq B''$. For each $\phi \in \mc{S}'$, define $\psi_{\circ\phi} \colon (0,1)^{2} \to (0,1)^{2}$ by
\[\psi_{\circ\phi}(x,y):=\langle \phi(x), (1-y)(g\circ \phi)(x) + y(h\circ \phi)(x)\rangle.\]
Then there clearly exists a positive real number $B'$, bounded effectively in $B$ and $r$, such that the set $\mc{S}:=\{\psi_{\circ\phi} \; | \; \phi \in \mc{S'}\}$ is an $r$-parameterization of $(g,h)_{(a,b)}$, with $\mc{S} \sset IRP(B')$ and $\#\mc{S} \leq B'$, as required.
\end{proof}

We now come to the proof of effective reparameterization for functions of two variables which are implicitly defined from restricted Pfaffian functions.
\begin{theorem} \label{thm:effrepara_2}
Let $n$, $r$ be non-negative integers and let $B$ be a positive real number. Suppose that $F\colon(0,1)^2\to (0,1)^n$ lies in the class $IRP(B)$. There exist a positive real number $B'$, which is bounded effectively in $B$, $r$ and $n$, and an $r$-reparameterization $\mathcal{S}$ of $F$ such that $\mc{S} \sset IRP(B')$ and the cardinality of $\mc{S}$ is bounded by $B'$.
\end{theorem}
\begin{proof}
Here the proof corresponds to that of the case (I)$_{1+1}$ in \cite{PilWil-06}, Section 5. We begin by applying the constant function version of Proposition \ref{prop:curvescells} (see Remark \ref{rmk:curves}) to the map $F$. This gives us a positive real number $B'_1$ which is effective in $B$, $r$ and $n$, non-negative integers $N, M_{0},\ldots,M_{N}$ bounded by $B'_1$, real numbers $\xi_{0}, \ldots, \xi_{N+1}$ with $0=\xi_{0}<\xi_{1}<\ldots<\xi_{N}<\xi_{N+1}=1$ and functions $\phi_{i,j}\colon C_i\to (0,1)$ lying in the class $IRP(B'_1)$, where $C_{i}=(0,1)\times (\xi_{i},\xi_{i+1})$ for $i=0,\ldots, N$, $j=0,\ldots,M_{i}$, such that, for each $i\in \{0,\ldots, N\}$ and every $y\in (\xi_{i},\xi_{i+1})$, the functions $\phi_{i,0}(\cdot,y),\ldots,\phi_{i,M_{i}}(\cdot,y)$ form an $r$-reparameterization of $F(\cdot,y)$.

Without loss of generality, let us fix $i\in \{0,\ldots, N\}$. We will now drop the index $i$ for clarity, and so we relabel the domain $C_{i}$ as $C=(0,1)\times (\nu,\xi)$; we then have $\phi_0,\ldots,\phi_M:C\to (0,1)$ such that, for all $y\in (\nu,\xi)$, the functions $\phi_0(\cdot,y),\ldots,\phi_M(\cdot,y)$ form an $r$-reparameterization of $F(\cdot,y)$. Define
\begin{eqnarray*}
{}^*F:C &\to& (0,1)^{2(M+1)}\\
\langle x,y\rangle& \mapsto & \langle\phi_0(x,y),\ldots,\phi_M(x,y),F(\phi_0(x,y),y),\ldots,F(\phi_M(x,y),y)\rangle.
\end{eqnarray*}
Then there is a positive real number $B'_{2}$ such that ${}^*F$ lies in the class $IRP(B'_2)$. Moreover ${}^*F$ satisfies the hypotheses of Lemma \ref{lem:PW4.4}. Therefore, applying this lemma with $k=r$, we obtain a positive real number $B_r$, a finite set $X_{r}\sset (\nu,\xi)$ and a set of functions $\mc{S}_{r}$ with the properties stated therein. Now consider the set of maps
\[\mc{S}'_{C}:=\{ \langle(\phi_{j})_{\psi},\psi\rangle \;|\; j = 0,\ldots,M \textrm{ and } \psi \in \mc{S}_{r} \}.\]
This set covers all of $C$ apart from at most the union of finitely many lines $(0,1)\times \{a\}$, for $a \in X_{r}$.

Now, for $y \in (0,1)$, define $\lmb_{y} \colon (0,1) \to (0,1)^{2}$ to be $\lmb_{y}(x) = \langle x,y\rangle$. The special case of Proposition \ref{prop:curvescells} (as mentioned in Remark \ref{rmk:curves}) will give a positive real number $B'_{3}$ bounded effectively in $B$ and $r$, and, for each $a \in X_{r}$, an $r$-reparameterization $\mc{T}_{a}$ of $F\circ \lmb_{a} \colon (0,1) \to (0,1)$ such that $\mc{T}_{a} \sset IRP(B'_3)$ and $\#\mc{T}_{a} \leq B'_{3}$, for each $a \in X_{r}$. Setting
\[ \mc{S}_{C}:= \mc{S}'_{C} \un \{\lmb_{a} \circ \tau \colon (0,1) \to (0,1)^{2} \; | \; a \in X_{r} \textrm{ and } \tau \in \mc{T}_{a}\}\]
provides an $r$-reparameterization of the restricted map $F\rst{C}$ that satisfies the required conditions.

We therefore finish by again applying the special case of Proposition \ref{prop:curvescells}, this time to obtain a positive real number $B'_4$ bounded effectively in $B$ and $r$, and $r$-reparameterizations $\mc{T}_{i}$ of $F\circ\lmb_{\xi_{i}} \colon (0,1) \to (0,1)$, for each $i = 1,\ldots, N$, such that $\mc{T}_{i} \sset IRP(B'_4)$ and $\#\mc{T}_{i} \leq B'_4$, for each $i = 1,\ldots, N$. Our desired $r$-reparameterization is then

\[ \mc{S}:= \Un^{N}_{i=0}\mc{S}_{C_{i}} \un \Un^{N}_{i=1}\mc{T}_{i}.\qedhere\]

\end{proof}

Finally we come to the proof of our main effective parameterization result, an effective parameterization theorem for surfaces implicitly defined from restricted Pfaffian functions. This is a straightforward corollary of the previous theorem.
\begin{theorem} \label{thm:effpara_2}
Let $n$, $r$ be non-negative integers, let $a$, $b$ be real numbers such that $0\leq a < b \leq 1$ and let $B$ be a positive real number. Suppose that $g, h \colon (a,b) \to (0,1)$ are functions lying in the class $IRP(B)$ with $g<h$. Suppose further that $F\colon (g,h)_{(a,b)}\to (0,1)^n$ lies in the class $IRP(B)$. There exists a positive real number $B'$ bounded effectively in $B$, $r$ and $n$, and an $r$-parameterization $\mathcal{S}$ of graph$(F)$ such that $\mathcal{S} \sset IRP(B')$ and the cardinality of $\mathcal{S}$ is bounded by $B'$.
\end{theorem}
\begin{proof}
In this instance, the proof corresponds to that of case (II)$_{2}$ in \cite{PilWil-06}, Section 5. Set $r \geq 1$. Apply Theorem \ref{thm:effpara_cell2} to $(g,h)_{(a,b)}$ to obtain a positive real number $B_1$ bounded effectively in $B$ and $r$, and an $r$-parameterization $\mc{S}_1$ of $(g,h)_{(a,b)}$ such that $\mc{S}_1 \sset IRP(B_1)$ and $\#\mc{S}_1 \leq B_1$. Now apply Theorem \ref{thm:effrepara_2} to obtain a positive real number $B_2$ bounded effectively in $B$, $r$ and $n$, and, for each $\phi \in \mc{S}_1$, an $r$-reparameterization $\mc{T}_{\phi}$ of $F\circ \phi \colon (0,1)^2 \to (0,1)^n$, such that $\mc{T}_{\phi} \sset IRP(B_2)$ and $\#\mc{T}_{\phi}\leq B_2$, for each $\phi \in \mc{S}_1$. Then we can define $\mc{S}$ to be $\{ \langle (\ph\circ\ps)(x,y), (F\circ\ph\circ\ps)(x,y)\rangle \; | \; \ps \in \mc{T}_{\ph}, \ph\in\mc{S}_1\}$. Clearly there exists a positive real number $B'$ with the required properties.
\end{proof}

\section{Counting results}
\label{sec:counting}
In this section we put our parameterization results to work and prove our counting results. These results make use of the following proposition from \cite{PilWil-06}. As there, an algebraic hypersurface of degree $d$ is the zero set of a non-zero polynomial of degree $d$, and, for a non-negative integer $n$, a positive real number $T$ and a set $Y \sset \bbR^{n}$, the set of rational points in $\bbQ^{n}$ lying on $Y$ of height at most $T$ is denoted $Y(\bbQ,T)$.

\begin{proposition}[\cite{PilWil-06}, Proposition 6.1] \label{prop:PWmain}
Let $k, n$ be non-negative integers with $k < n$. For each positive integer $d$, there exists a non-negative integer $r = r(k,n,d)$ and positive constants $\e(k,n,d), C(k,n,d)$ with the following property. For any $C^{r}$-map $\ph \colon (0,1)^{k} \to \bbR^{n}$ with $\md{\ph^{(\a)}(x)} \leq 1$, for all $x \in (0,1)^{k}$ and all $\a \in \bbN^{k}$ with $\md{\a} \leq r$, and for all $T\geq 1$, the set Im$(\ph)(\bbQ,T)$ is contained in the union of at most $C(k,n,d)T^{\e(k,n,d)}$ algebraic hypersurfaces of degree at most $d$. Furthermore, $\e(k,n,d) \to 0$ as $d \to \infty$.

\end{proposition}

Importantly for us, it is know that the  $r(k,n,d)$, $\e(k,n,d)$ and $C(k,n,d)$ in this statement can be effectively computed from $k,n$ and $d$. This follows from the proof of \ref{prop:PWmain}; see in particular \cite[4.1, 4.2]{Pil-04}. We will use this effectivity below.

In order to prove our main results in this section, we will use the following proposition, a particular case of Proposition 5.3 from \cite{JonTho-12}.

\begin{proposition}[\cite{JonTho-12}, Proposition 5.3]\label{prop:JT12_zeroest}
Let $B$ be a positive real number and suppose that $f\colon(0,1)^2\to (0,1)$ is a function lying in the class $IP(B)$. Let $X$ be the graph of $f$. There exist positive integers $N=N(B)$, $\g=\g(B)$ and $c_{1}$, and a polynomial $Q \colon \bbR \to \bbR$ over $\bbR$ of degree $N$ with coefficients depending only on $B$, such that, for all $T\geq 1$ and for all positive integers $d$, if $P\colon \bbR^{3} \to \bbR$ is a polynomial of degree $d$, then
\[
\# (\tr{(X \cap V(P))})(\bbQ,T) \leq c_{1}Q(d)(\log T)^{\g}.
\]
\end{proposition}
In \cite{JonTho-12} we did not address effectivity for this statement but, upon inspecting the proof of this result, it is possible to see that the constants $N=N(B)$, $\g=\g(B)$ and $c_{1}$ may be found effectively, as well as an effective bound on the coefficients of the polynomial $Q$.

We now present our first counting result, for those surfaces implicitly defined from restricted Pfaffian functions whose base is an open cell lying inside the box $(0,1)^{2}$.

\begin{theorem} \label{thm:main_surface_res_inbox}
Let $a$, $b$ be real numbers such that $0\leq a < b \leq 1$ and let $B$ and $\e$ be positive real numbers. Suppose that $g, h \colon (a,b) \to (0,1)$ are functions lying in the class $IRP(B)$ with $g<h$. Suppose further that $F\colon (g,h)_{(a,b)}\to \bbR$ lies in the class $IRP(B)$. There exists a positive real number $c$, bounded effectively in $B$ and $\epsilon$, such that, for all $T\geq 1$,
\[
\#\text{graph}(F)^{\text{trans}} (\bbQ,T) \leq cT^{\epsilon}.
\]
\end{theorem}
\begin{proof}
We start by applying Proposition \ref{prop:PWmain} with $k=2$, $n=3$ and $d$ an integer large enough that $\epsilon(2,3,d)\leq \epsilon/2$ (so $d$ is effective in $\e$). Let $r=r(2,3,d)$ and $C=C(2,3,d)$. As remarked above, $r$ and $C$ are therefore both effective in $\epsilon$.

Define the function $G\colon (g,h)_{(a,b)}\to \bbR$ by \[G(x,y) = F(x,y)(F(x,y)^2 - 1).\] Since $F$ lies in $IRP(B)$, there exists a positive real number $B_1$ which is bounded effectively in $B$ such that $G$ lies in $IRP(B_1)$. By repeated application of Lemma \ref{lem:effzerosetdecomp} (see Remark \ref{rmk:effzerosetdecomp}), we see that there is a positive real number $B_2$, bounded effectively in $B$, such that $V(G)$ can be decomposed into at most $B_2$ graphs of functions lying in $IRP(B_2)$, points and vertical lines. These give us a cell decomposition $\mc{C}$ of $(g,h)_{(a,b)}$ such that $F$ restricted to each of the open cells in $\mc{C}$ takes values in exactly one of $(-\infty,-1), (-1,0), (0,1), (1, \infty)$, and a positive real number $B_3$ effective in $B$ such that the cardinality of $\mc{C}$ is bounded by $B_3$.

For each open cell $\mfk{C}$ in $\mc{C}$, compose $F\rst{\mfk{C}}$ with whichever of the four inversion maps $x \mapsto \pm x^{\pm 1}$ is appropriate to obtain a function $\widetilde{F}_\mfk{C} \colon \mfk{C} \to (0,1)$ such that, for each $T \geq 1$, $\#\text{graph}(\widetilde{F}_\mfk{C})^{\text{trans}} (\bbQ,T) = \#\text{graph}(F\rst{\mfk{C}})^{\text{trans}} (\bbQ,T)$. Such an open cell $\mfk{C}$ is of the form $(\widetilde{g},\widetilde{h})_{(\widetilde{a},\widetilde{b})}$, for real numbers $\widetilde{a}$, $\widetilde{b}$ such that $a\leq \widetilde{a}< \widetilde{b} \leq b$ and functions $\widetilde{g}, \widetilde{h}\colon (\widetilde{a},\widetilde{b}) \to (0,1)$ in $IRP(B_2)$, and there exists a positive real number $B_4$ which is bounded effectively in $B$ such that $\widetilde{F}_{\mfk{C}}$ lies in $IRP(B_4)$. Therefore, we may apply Theorem \ref{thm:effpara_2} to $\widetilde{F}_\mfk{C}$ to obtain a positive real number $B_5$, effective in $B$ and $r$, and hence in $B$ and $\epsilon$, and an $r$-parameterization $\mathcal{S}$ of $\text{graph}(\widetilde{F}_\mfk{C})$ lying in $IRP(B_5)$ with cardinality bounded by $B_5$.

Let $T\geq 1$. By Proposition \ref{prop:PWmain}, $\text{graph}(\widetilde{F}_\mfk{C})(\bbQ,T)$ is contained in the union of at most $B_5\cdot C\cdot T^{\epsilon/2}$ algebraic surfaces of degree at most $d$, for each open cell $\mfk{C}$ in $\mc{C}$.

Fix such a surface, $V(P)$, say. We now apply Proposition \ref{prop:JT12_zeroest} to obtain effective constants $N=N(B)$, $\g=\g(B)$ and $c_{1}$, and a polynomial $Q \colon \bbR \to \bbR$ over $\bbR$ of degree $N$ with coefficients bounded effectively in $B$, such that
\[
\# (\text{graph}(\widetilde{F}_\mfk{C})\cap V(P))^{\text{trans}}(\bbQ,T) \leq c_{1} Q(d) (\log T)^{\gamma},
\]
for each open cell $\mfk{C}$ in $\mc{C}$. Since $Q(d)$ is a polynomial in $d$ with coefficients which are bounded effectively in $B$, $Q(d)$ is in fact just a constant effective in $B$ and in $\epsilon$.

Now take $T_{0}$ such that $(\log T)^{\gamma}< T^{\epsilon/2}$ for all $T \geq T_{0}$, which we can do effectively in terms of $\gamma$ and $\epsilon$, and hence in terms of $B$ and $\epsilon$. Hence, for $T \geq T_{0}$, we have that, for each open cell $\mfk{C}$ in $\mc{C}$,
\[
\# (\text{graph}(\widetilde{F}_\mfk{C})\cap V(P))^{\text{trans}}(\bbQ,T) \leq c_{2} T^{\epsilon/2},
\]
where $c_{2}$ is effective in terms of $B$ and $\epsilon$. Therefore
\[
\#\text{graph}(F)^{\text{trans}} (\bbQ,T) \leq B_{3} \cdot B_{5}\cdot C\cdot c_{2}\cdot T^{\epsilon},
\]
for $T \geq T_{0}$. Finally, increasing the constant at the front to compensate for $T_{0}$, which we can do effectively, we see that a bound of the required form holds for all $T\geq 1$.
\end{proof}

We now use this result to obtain a counting statement in the unrestricted setting. Given a function lying in $IP$, we note that it is possible to shrink its domain in such a way that the restricted function obtained lies in $IRP$, the restricted domain approximates the original domain to an arbitrary extent, and the complexities of the functions involved in defining these two domains are the same. Applying Theorem \ref{thm:main_surface_res_inbox} to the restricted function obtained in this way, we get a bound on rational points for its graph that does not depend on the particular domain chosen. Since all rational points of a given height that lie on the graph of the unrestricted function also lie on the graph of such a restricted function, we may choose a suitable restriction of this kind and hence obtain a bound on rational points in the unrestricted setting.

\begin{theorem} \label{thm:main_surface_unres}
Let $a$, $b$ be real numbers such that $0\leq a < b \leq 1$ and let $B$ and $\e$ be positive real numbers. Suppose that $g$ is either a function $g:(a,b)\to (0,1)$ lying in $IP(B)$, or is the constant function $0$ defined on $(a,b)$, and suppose that $h$ is either a function $h:(a,b)\to(0,1)$ lying in $IP(B)$, or is the constant function $1$ defined on $(a,b)$. Suppose moreover that $g<h$, and that $F:(g,h)_{(a,b)}\to \bbR$ lies in $IP(B)$. 
There exists a positive real number $c$, bounded effectively in $B$ and $\epsilon$, such that, for all $T\geq 1$,
\[
\#\text{graph}(F)^{\text{trans}} (\bbQ,T) \leq cT^{\epsilon}.
\]
\end{theorem}
\begin{proof}
Fix $T\ge 1$ and $\epsilon>0$. Since there are only finitely many rational points of height at most $T$ inside $(g,h)_{(a,b)}$, we may find a positive real number $\delta$ small enough that $a+\delta<b-\delta$, such that $g+\delta<h-\delta$ holds on $(a+\delta,b-\delta)$ and such that all rational points of height at most $T$ inside $(g,h)_{(a,b)}$ in fact lie inside the cell $C_\delta=(g+\delta,h-\delta)_{(a+\delta,b-\delta)}$. We then have
\[\#\text{graph}(F\rst{C_{\delta}})^{\text{trans}} (\bbQ,T) = \#\text{graph}(F)^{\text{trans}} (\bbQ,T) .\]
The function $F\rst{C_\delta}$ is in $IRP(B)$ by Corollary  \ref{cor:resIP=IRP}. So we can apply Theorem \ref{thm:main_surface_res_inbox}. The constant $c$ this provides is independent of $\delta$, and so independent of $T$, and this proves the theorem.
\end{proof}

Finally, applying the usual inversion process in combination with effective monotonicity (Proposition \ref{prop:effmon}) we obtain the most general form of our result. 
\begin{corr} \label{cor:main_surface_unres}
Let $(a,b)$ be an interval in $\bbR$, with $a \in \{-\infty\}\cup\bbR$ and $b\in \bbR\cup\{+\infty\}$, and let $B$ and $\e$ be positive real numbers. Suppose that $g$ is either a function $g:(a,b)\to \bbR$ lying in $IP(B)$, or is the constant function $-\infty$ defined on $(a,b)$, and suppose that $h$ is either a function $h:(a,b)\to\bbR$ lying in $IP(B)$, or is the constant function $+\infty$ defined on $(a,b)$. Suppose moreover that $g<h$, and that $F:(g,h)_{(a,b)}\to \bbR$ lies in $IP(B)$.
There exists a positive real number $c$, bounded effectively in $B$ and $\epsilon$, such that, for all $T\geq 1$,
\[
\#\text{graph}(F)^{\text{trans}} (\bbQ,T) \leq cT^{\epsilon}.
\]
\end{corr}

\newcommand{\SortNoop}[1]{}\def\cprime{$'$} \def\cprime{$'$}

\end{document}